\documentclass{article}
\usepackage[colorlinks=true,citecolor=blue]{hyperref}
\usepackage{amsmath, amssymb, amsfonts, amsthm, enumerate, color, algorithm, algpseudocode, graphicx, multirow}
\usepackage{graphicx}
\usepackage{epstopdf}

\newtheorem{theorem}{Theorem}[section]
\newtheorem{corollary}{Corollary}[section]
\newtheorem{lemma}{Lemma}[section]

\theoremstyle{definition}

\newtheorem{assumption}{Assumption}[section]
\numberwithin{equation}{section}

\newcommand{\cmark}{\textbf{\checkmark}}
\newcommand{\xmark}{\textbf{\texttimes}}

\title{Oblivious Stochastic Composite Optimization}
\author{Cl\'ement Lezane, Alexandre d'Aspremont} 
\begin{document}
\maketitle

\begin{abstract}
In stochastic convex optimization problems, most existing adaptive methods rely on prior knowledge about the diameter bound $D$ when the smoothness or the Lipschitz constant is unknown. This often significantly affects performance as only a rough approximation of $D$ is usually known in practice. Here, we bypass this limitation by combining mirror descent with dual averaging techniques and we show that, under oblivious step-sizes regime, our algorithms converge without any prior knowledge on the parameters of the problem. We introduce three oblivious stochastic algorithms to address different settings. The first algorithm is designed for objectives in relative scale, the second one is an accelerated version tailored for smooth objectives, whereas the last one is for relatively-smooth objectives. All three algorithms work without prior knowledge of the diameter of the feasible set, the Lipschitz constant or smoothness of the objective function. We use these results to revisit the problem of solving large-scale semidefinite programs using randomized first-order methods and stochastic smoothing. We extend our framework to relative scale and demonstrate the efficiency and robustness of our methods on large-scale semidefinite programs.
\end{abstract}

\section{Introduction}
We address stochastic convex optimization problems where regularity parameters are unknown a priori. Our goal is to solve the population loss minimization problem
\begin{equation} 
\label{eq-population-loss}
\min_{x \in \mathcal{X}} \quad  F(x) : =\mathbb{E}_{\xi} [ f_{\xi} ( x )]
\end{equation}
with $\xi$ being a random variable, $f_\xi$ being convex functions and $\mathcal{X}$ being a convex set. In the adaptive setting (cf. \cite{levy2018online,cutkosky2019anytime}) three key assumptions are fundamental in all our problem settings. We first assume access to an unbiased stochastic oracle $\nabla f_{\xi}(x) $ estimating $\nabla F(x)$, i.e. 
\[
\forall x \in \mathcal{X}, \quad \mathbb{E}_{\xi} [\nabla f_\xi(x)] = \nabla F(x)
\] 
which provides noisy gradient information for the population loss. The second assumption is the availability of a strictly convex function $H$ that serves as a regularizer. For example, we could define $H$ as the squared norm $H = \|\cdot \|_p^2$ for $\ell_p$ spaces when $1<p \leq 2$. The last assumption concerns various bounds on the problem depending on the setting, that will be detailed in the next subsection.

\subsection{Problem settings}
We will focus on oblivious algorithms in multiple problem settings, namely  optimization in relative scale, standard smoothness and relative smoothness assumptions. In what follows, we will write the Bregman divergence of a continuously derivable function $f$ by $D^f$, 
\[ \forall x,y \in \mathcal{X}, \quad D^{f}(x,y) := f(x) - f(y) - \langle \nabla f(y), x-y  \rangle\]

\textbf{Optimization in relative scale.} In the relative scale setting, we extend results from \cite{nesterov2023randomized} which impose a relative constraint on the stochastic oracle and a lower bound on $F$, written
\begin{equation}
\label{def-relative-scale}
\begin{aligned}
\forall x \in \mathcal{X}, \quad & \mathbb{E}_{\xi}  \Big[ \| \nabla f_{\xi}(x)\|^2 \Big] \leq  \mathcal{M} F(x) \\
\exists H, \exists \hat{x}_0 \in \mathcal{X}, \forall x \in \mathcal{X}, \quad & F(x) \geq \Gamma  \Big( H(x) -  H(\hat{x}_0)  \Big)
\end{aligned}
\end{equation}
where $\mathcal{M}, \Gamma >0$ are positive constants and $\hat{x}_0 \in \mathcal{X}$ a fixed point in the space. In the original setting \cite{nesterov2023randomized}, the condition is satisfied with $H(x) := \| x - \hat{x}_0 \|_2^2$ being the Euclidean norm. As a consequence, the objective function $F$ must be positive. The authors used prior knowledge on $(\mathcal{M}, \Gamma,\hat{x}_0)$ and designed an algorithm that guarantees the following convergence rate,
\begin{equation}
    \mathbb{E}[F(x_{T+1})] - F_\star \leq  F_\star \sqrt{\frac{2 \mathcal{M}}{\Gamma T}}
\end{equation}
where $x_{T+1}$ is the final output of their algorithm.

\textbf{Minimizing smooth convex functions.} For the standard smooth case, we require the objective function $F$ to be $L$-smooth with respect to an arbitrary norm $\| \cdot\|$ and the stochastic noise to have a finite variance $\sigma^2$, i.e.
\begin{equation}
\label{def-smooth}
\begin{aligned}
\forall x, y \in \mathcal{X}, \quad & D^{F}(x,y) \leq \frac{L}{2} \|x-y\|^2 \\
\forall x \in \mathcal{X}, \quad & \mathbb{E}_{\xi} \Big[ \| \nabla f_{\xi}(x) - \nabla F(x) \|_{\ast}^2 \Big] \leq \sigma^2 
\end{aligned}
\end{equation}
where $\| \cdot \|_\ast$ is the dual norm of $\| \cdot \|$. In that case, we can use the accelerated framework in \cite{Lan11} to get
\begin{equation}
 \mathbb{E} \Big[ F(x_{T+1})- F_\star \Big] \lesssim \frac{L D^2}{T^2 } + \frac{\sigma D}{\sqrt{T}},  
\end{equation}
where $x_{T+1}$ is the output of the accelerated algorithm. The convergence rate is optimal for the smooth case \cite{NY83,Lan11}, but the algorithm requires prior knowledge on $L$ and $D$.

\textbf{Optimization for relatively-smooth functions.} In the relatively-smooth setting \cite{Lu2016RelativelySC} \cite{Dragomir2021FastSB}, the smoothness is related to a convex function $H$ instead of a norm. We note that, in the relatively smooth case, $H$ is only required to be strictly convex and twice continuously differentiable. We assume that the stochastic oracle satisfies
\begin{equation}
\begin{aligned}
\forall x,y \in \mathcal{X}, \quad D^{F}(x,y) \leq L_\star D^{H}(x,y)  \\
\forall \eta < \eta_0, \quad \mathbb{E}_\xi \left[ D^{H_\star}\Big( \nabla H(x) - 2\eta \nabla f_{\xi}(x_F), \nabla H(x)\Big) \right] \leq 2\eta^2 \sigma^2 
\end{aligned}
\end{equation}

For that relative smooth case, we can use the mirror descent algorithm in \cite{Dragomir2021FastSB} to get
\begin{equation}
\mathbb{E} \Big[ F(x_{T+1})- F_\star \Big] \lesssim \frac{\sigma D_0}{\sqrt{T}} 
\end{equation}
where $x_{T+1}$ is the final output of the algorithm. Note that we need an extra assumption in relative smoothness to show convergence.
\begin{assumption}
For all $y \in \mathbb{R}^d$, the following problem
\[ \min_{x \in \mathcal{X}} H(x) - \langle y, x \rangle \]
has an unique solution in the interior of $\mathcal{X}$.
\end{assumption}

\subsection{Related work}

Previous methods required prior knowledge of different parameters. Having no or partial information on the problem's parameters has a direct impact on the performance of most first order algorithms, as they are often used to select step-sizes for example. Even if one does have a rough estimate of some parameters (such as the Lipschitz constant), over or under-estimation makes these algorithms slow and inefficient. The key challenge is developing algorithms that can either handle parameter uncertainty or learn these parameters from scratch. 

\subsubsection{Polyak step-size and Line search.}
For the non-smooth case, the Polyak step-size offers a viable alternative \cite{polyak}, but this method requires knowledge of $F_\star := \min F$. For the smooth case, line search methods help if the gradient is known exactly \cite{linesearchNesterov,linesearch14}, but designing a stochastic counterpart of line search is still an open question.

\subsubsection{Grid search.} 
When direct evaluation of the objective function $F$ is feasible, an alternative strategy involves running multiple parallel algorithm instances, each configured with different parameter settings. The final solution is selected as the output that achieves the minimum value of $F$. This Grid seach approach was originally proposed by Nemirovski for smooth problems where the gradient is deterministic and known, recommending $\log(T)$ parallel sessions to achieve optimal performance guarantees. We have extended this idea for the stochastic case when evaluating the objective is still possible but might be costly. More details can be found in Section \ref{section-parallelization}.

\subsubsection{Adagrad and online algorithms.}
For the smooth case, different variations on the Adagrad step-size \cite{adagrad} prove useful if we have prior knowledge on the diameter $D$ of the feasible set $\mathcal{X}$, but $L$ remains unknown. Writing $g_t = \nabla f_{\xi_t}(x_t)$ the stochastic oracle at $x_t$, the Adagrad step-size becomes
\[ 
\eta_t = \frac{2D}{\sqrt{D^2 + \sum_{\tau=1}^{t-1} \alpha_\tau \| g_\tau \|^2}}
\]
We mention here a general adaptive method \cite{levy2018online} that works for both non-smooth and smooth cases, and has inspired our proof technique. In the smooth case, the convergence rate becomes
\begin{equation} 
 \mathbb{E} \Big[ F(x_{T+1})- F_\star \Big] \lesssim \frac{LD^2}{T} + \frac{\sigma D}{\sqrt{T}},  
\end{equation}
which is the standard non-accelerated rate. In comparison, the optimism in online learning scheme \cite{cutkosky2019anytime} provides a sharper accelerated rate written
\begin{equation} 
 \mathbb{E} \Big[ F(x_{T+1})- F_\star \Big] \lesssim \frac{LD^2}{T\sqrt{T}} + \frac{\sigma D}{\sqrt{T}},   
\end{equation}

but the optimism algorithm requires implicitly prior knowledge on $D$. In the same way, another rate $O \Big( \frac{L D^2\log( \mathcal{M} T)}{T^2} + \frac{\sigma D\log( \mathcal{M} T) }{\sqrt{T}} \Big) $ is derived in \cite{cutkosky2019anytime}; however, this method requires an almost sure bound on the stochastic gradients, aside from knowledge of $D_0$.

More bounds have been provided in later works \cite{pmlr-v119-joulani20a}, but they require additional conditions or have a convergence rate explicitly depending on the dimension $d$, so we omit them here. On the other hand, a parameter free online mirror descent algorithm has been proposed in \cite{jacobsen2022parameterfree}, assuming extra conditions on the second derivative of the proximal function, which do not cover the smooth case described here.

\subsection{Our contributions}
Our objective here is to develop algorithms capable of handling diverse stochastic convex optimization scenarios where key problem parameters remain unknown.

\subsubsection{Composite Mirror Descent with Oblivious Step-sizes.} 
When confronting optimization problems with limited knowledge of regularity parameters, we introduce a regularization framework $H$ based on complementary composite stochastic optimization techniques \cite{Nesterov12,diakonikolas2023complementary,daspremont2022optimal}. We then construct a family of oblivious step-sizes $\{\alpha_t, \beta_t, \gamma_t\}$ that enable convergence without requiring prior knowledge of Lipschitz or smoothness constants.

\subsubsection{Applications Across Different Problem Classes.} 
Our algorithmic framework applies to three important non-smooth optimization settings: relative-scale precision targets \cite{nesterov2023randomized}, smooth objectives \cite{GL12}, and relatively-smooth functions \cite{Dragomir2021FastSB}. Relative-scale optimization measures complexity in terms of digits of precision gained rather than absolute precision targets, often yielding substantial computational savings. Relative smoothness generalizes the objective function's geometric properties by relating them to arbitrary regularization functions instead of standard norms.

\subsubsection{Practical Impact and Theoretical Foundations.} 
Our oblivious algorithms combine ease of implementation with strong empirical performance while providing interesting theoretical convergence guarantees. We summarize our contribution with respect to existing methods in Table~\ref{table comparison}.

\begin{table}[ht]
\centering
\begin{tabular}{|c|c|c|c|c|} 
\hline
  & \multicolumn{2}{|c|}{Oblivious} & &  \\
 \hline
 Class & $D$   & $\mathcal{M}/L/L_\star$ & Unbounded noise & Convergence rate \\ [0.5ex] 
 \hline \hline
 \multirow{2}{*}{Relative scale}
   & \xmark  & \xmark  &   \cmark & $  O \Big( \sqrt{\frac{\mathcal{M}  }{\Gamma  }} \frac{F_\star}{\sqrt{T}} \Big) $ 
   \cite{levy2018online}  \\
& \cmark &  \cmark  & \cmark  & $ \Tilde{O} \left( \sqrt{\frac{\mathcal{M}  }{\Gamma  }} \frac{F_\star}{\sqrt{T}} + \frac{1} {T^{n/2}} \right) $     \\ 
\hline
  \multirow{4}{*}{$L$-smooth}
  &  \xmark  & \cmark  & \cmark  & $ O \Big( \frac{L D^2}{T  } + \frac{\sigma D}{\sqrt{T}} \Big) $ \cite{levy2018online} \\
  &  \xmark  & \cmark  & \cmark  & $ O \Big( \frac{L D^2\log(T)^2}{T  } + \frac{\sigma \log(T) }{\sqrt{T}} \Big) $  \cite{cutkosky2019anytime}\\
    &  \xmark  & \cmark  & \xmark  & $ \Tilde{O} \Big( \big(\frac{L D^2}{T^2 } + \frac{\sigma D}{\sqrt{T}}\big) \Big) $  \cite{cutkosky2019anytime} \\
      &  \cmark  & \cmark & \cmark    &$ \Tilde{O} \Big( \frac{\sigma  D }{\sqrt{T}} + \frac{1}{T^n} \Big)^{(\star)}  $  \\
 \hline
$L_{\ast}$-relatively-smooth  & \xmark  & \xmark &  \cmark   & $ O \Big( \frac{\sigma D_0}{\sqrt{T}} \Big) $  \cite{Dragomir2021FastSB}  \\ 
$L_{\ast}$-relatively-smooth  & \cmark  & \cmark &  \xmark   & $ \Tilde{O} \Big(  \frac{\sigma  D_0 }{\sqrt{T}} + \frac{1}{T^{n}} \Big)^{(\star)}$    \\ 
\hline
\end{tabular}\vskip 1ex
\caption{Summary of convergence rates with limited prior knowledge on parameters (absolute constants omitted).  The exponent $n$ is the degree of the oblivious step size which can be chosen in an arbitrary way. $(\star)$ are the contributions of this paper. Results given with $\Tilde{O}$ have an additional polylog factor. \label{table comparison}}
\end{table}

\subsection{Applications to semidefinite programming} 
Consider the problem of minimizing the maximum eigenvalue of a symmetric matrix $x$, over a convex set $\mathcal{X}\subseteq S^d$, written
\begin{equation} \label{lambda max}
\min_{x \in \mathcal{X}} \quad   \lambda_{\max}( x ), 
\end{equation}
in the variable $x \in S^d$ where $S^d$ is the set of symmetric real $d\times d$ matrices.

While semidefinite programs (SDP) can be efficiently solved by interior point methods, their algorithmic complexity becomes prohibitive for large-scale problems. First-order methods offer a viable alternative, albeit with a potential sacrifice in accuracy. For many large-scale applications however, achieving an approximate solution within an acceptable range of the optimum is often sufficient. This has fueled a growing interest in the development of efficient first-order algorithms to handle large-scale semidefinite programs \cite{daspremont2014stochastic,Yurtsever_2021}.

Three common approaches to solve problem \eqref{lambda max} will serve as a starting point for our work. The first approach is the celebrated mirror prox algorithm \cite{Nemirovski04}, whose main limitation is the requirement of a full spectral decomposition of a symmetric matrix at each iteration.

A second approach is based on a technique known as {\em stochastic smoothing}. 
In \cite{daspremont2014stochastic}, for instance, the authors use stochastic smoothing with $k$ random rank one perturbations and target precision $\epsilon >0$ to produce a smooth approximation $F_{k,\epsilon}$, written as follows
\begin{equation}
  F_{k,\epsilon}(x) := \mathbb{E}_{z_i \sim \mathcal{N}(0,I_d)} \left[\max _{i=1, \ldots, k} \lambda_{\max }\left(x+(\epsilon / d) z_i z_i^T\right)\right]. 
\end{equation}
More recently, a different type of stochastic smoothing based on matrix powers with exponent $p\geq 1$ was formulated in \cite{nesterov2023randomized} as
\begin{equation}
  F_p(x) := \mathbb{E}_{u \in \mathcal{U}([0,1]^d)} \Big[ \left\langle x^p u, u\right\rangle^{1 / p} \Big]. 
\end{equation}
Computing the minimum of an approximate objective function, as above, can be more effective depending on  problem structure. A more detailed comparison is given in the introduction of \cite{daspremont2014stochastic} and in \cite{nesterov2023randomized}. For solving large scale SDPs, we will call $F$ the approximate objective given by a form of stochastic smoothing (it could be $F_{k,\epsilon}$ or $F_p$) and we will suppose $F$ convex (which is the case above).

\section{Preliminaries}
We recall the framework of the Stochastic Convex Optimization problem
\begin{align*} 
\min_{x \in \mathcal{X}} \quad F(x) =   \mathbb{E}_{\xi} [ f_{\xi} ( x )].
\end{align*}
We suppose that we have a stochastic gradient oracle $\nabla f_\xi(x)$, and define the composite mirror descent step as
\begin{equation} 
x_{t+1} = \arg \min_x \{ \alpha_t \langle \nabla f_{\xi_t}(x_t),x \rangle + \beta_t H(x) + \gamma_t D^{H}(x,x_t) \} 
\end{equation}
where $\{ \alpha_t, \beta_t, \gamma_t\}_{t \geq 1}$ is our sequence of step-sizes. Inspired by previous works on the composite setting \cite{Beck:2009,Nesterov:2013,GL12,daspremont2022optimal}, we focus on the following regime
\begin{equation}
\label{oblivious-stepsize}
\begin{aligned}
\begin{cases}
\gamma_{t+1} -\gamma_t & \leq \beta_t \\
\gamma_{t}/\alpha_t & \leq \gamma_{t+1}/\alpha_{t+1}. \\
\end{cases}
\end{aligned}
\end{equation}

Particularly, when the information on different parameters of the problem is not available, we will consider general oblivious step-sizes that do not take any parameters into account. More precisely, we use a family of polynomial step-sizes $\alpha_t \propto t^{n}, \beta_t \propto \mu t^{n}, \gamma_t \propto \mu t^{n+1}/(n+1)$ with a constant $\mu >0$ and a degree $n \geq 0$. As we will see later in Section~\ref{section-composite-mirror descent}, the parameters $(\mu,n)$ are oblivious to the instance, as they are selected a priori, in a problem-independent way. 

\subsection{Useful Lemmas}
We introduce important lemmas in this subsection. First, we present an extension of the three-points identity.
\begin{lemma}  \label{lem:prox}
Let $f$ be a convex function and $\nu$ continuously differentiable. If we consider 
\[ 
u^{\star} = \arg\min_{u\in {\mathcal{X}}} \{f(u) + D^{\nu}(u,y)\}, 
\]
with the Bregman divergence $D^{\nu}(a,b) := \nu(a) - \nu(b) -\langle \nabla \nu (b),a-b\rangle$ then for all $u$,
\[
f(u^{\star}) + D^{\nu}(u^\star,y) + D^{f}(u,u^\star) \leq f(u) +  D^{\nu}(u,y) - D^{\nu}(u,u^\star). 
\]\end{lemma}

\begin{proof}
From the first order optimality conditions, for all $u\in{\mathcal{X}}$: 
\[ \langle \nabla f(u^\star) + \nabla_u D^{\nu}(u^{\star},y) , u - u^{\star}\rangle \geq 0,\] with the gradient taken with respect to the first entry. We also apply the three-points identity
\[\langle \nabla D^{\nu}(u^{\star},y) , u - u^{\star}\rangle = D^{\nu}(u,y)- D^{\nu}(u,u^{\star})- D^{\nu}(u^{\star},y).\]
Then
\begin{align*}
  f(u) - f(u^{\star}) - D^{f}(u,u^{\star}) & = \langle \nabla f(u^\star) , u - u^{\star}\rangle \\
    & \geq D^{\nu}(u,u^{\star})- D^{\nu}(u,y)  +D^{\nu}(u^{\star},y).
\end{align*}
\end{proof}
We note that Lan's work \cite{GL12} follows a similar route, except for a negative Bregman divergence term upper bounded by zero, that we maintain here. We now introduce an extension of Lemma 4 in \cite{Dragomir2021FastSB} in the case of relative smoothness.

\begin{lemma}  \label{lem:prox-free}
Let $f$ be a convex function and $\nu$ continuously differentiable. If we consider  
\[ 
u^{\star} = \arg\min_{u\in {\mathcal{ X}}} \{f(u) + D^{\nu}(u,y)\}, 
\]
with the Bregman divergence $D^{\nu}(a,b) := \nu(a) - \nu(b) -\langle \nabla \nu(b),a-b\rangle$ and $u^\star$ is in the interior of $\mathcal{X}$, then for all $u \in \mathcal{X} $ ,
\[
f(u^{\star}) - f(u) + D^{\nu}(u^\star,y) + D^{f}(u,u^\star) =   D^{\nu}(u,y) - D^{\nu}(u,u^\star). 
\]
\end{lemma}
\begin{proof}
The proof is similar as before with only difference of $u^\star$ being in the interior of $\mathcal{X}$, 
\[\nabla f(u^\star) + \nabla D^{\nu}(u^{\star},y) = 0. \] The rest follows.
\end{proof}

For the relative smooth case, it is also to show the Lemma 3 from \cite{Dragomir2021FastSB} on thes property and cocoercivity of the Bregman divergence.

\begin{lemma} \label{lem:cocoercivity}
Consider two convex functions $F,H:\mathcal{X} \xrightarrow{} \mathbb{R}$ where $F$ is $L_\star$-relatively smooth $H$. For any $ 0< \eta< \frac{1}{L_\star}$ and $x,y \in int(\mathcal{X})$, we have
\[   D^{H_\star}\Big( \nabla H(x) - \eta(\nabla F(x) - \nabla F(y)) , \nabla H(x) \Big)  \leq \eta D^F(x,y) \]
where $H_\star$ is the convex conjugate of $H$.
\end{lemma}
\begin{proof}
For all $y \in int(\mathcal{X})$, we note
\begin{align*}
D^F( \cdot ,y) = F(\cdot) - F(y) - \langle \nabla F(y), \cdot - y\rangle
\end{align*}
Therefore, for any $\eta \leq \frac{1}{L_\star}$, $y,x,u \in int(\mathcal{X})$ we know by the definition of relative smoothness
\begin{align*}
D^F(u,x) = D^{D^F(\cdot,y)}(u,x) = D^F(u,y) - D^{F}(x,y) - \langle \nabla_x D^F(x,y), u-x \rangle \leq \frac{1}{\eta} D^H(u,x)
\end{align*}
We notice that the function $Q_{y,x}: u \xrightarrow{} \frac{1}{\eta} D^H(u,x) + D^{F}(x,y) + \langle \nabla_x D^F(x,y), u-x \rangle $ is convex, greater than $D^F(u,y)$. We note $u^{+}$ as the minimum of $Q_{y,x}$, the gradients cancel at this point:
\begin{align*}
\frac{1}{\eta} \Big( \nabla H(u^{+}) - \nabla H(x) \Big) + \nabla_x D^F(x,y) & = 0 \\
\Leftrightarrow \nabla H(u^{+}) = \nabla H(x) - \eta \nabla_x D^F(x,y) & = \nabla H(x) - \eta \nabla F(x) + \eta \nabla F(y) 
\end{align*}
By the previous inequality, we know that for any $\eta \leq \frac{1}{L_\star}$,
\begin{align*}
0 \leq D^F(u^{+},y) \leq & Q_{y,x}(u^{+}) \\
= & \frac{1}{\eta} D^H(u^{+},x) + D^{F}(x,y) + \langle \nabla_x D^F(x,y), u^{+}-x \rangle \\
= & \frac{1}{\eta} D^H(u^{+},x) + D^{F}(x,y) + \frac{1}{\eta} \langle  \nabla H(x) - \nabla H(u^{+}) , u^{+}-x \rangle \\
= & D^{F}(x,y) - \frac{1}{\eta} D^H(x, u^{+}) 
\end{align*}
We conclude by noticing that
\begin{align*}
D^H(x, u^{+}) &  = D^{H_\star}(\nabla H(u^{+}),\nabla H(x))  \\
& = D^{H_\star}(\nabla H(x) - \eta \nabla F(x) + \eta \nabla F(y) ,\nabla H(x)).
\end{align*}

\end{proof}

\section{Composite Mirror Descent}
\label{section-composite-mirror descent}
Our approach combines stochastic mirror descent with dual averaging to minimize convex functions without prior knowledge about the problem. The general analysis takes inspiration from existing methods like minimization in relative scale \cite{nesterov2023randomized}, the composite setting \cite{daspremont2022optimal}, and relative smoothness \cite{Dragomir2021FastSB}. In the following part, we will present three variations of our algorithm in these various settings.

\subsection{Relative scale setting}
We recall that in the relative scale setting, there exist constants $\Gamma, \mathcal{M} > 0 $ and a fixed point $\hat{x}_0$ such that for all $x \in \mathcal{X}$,
\begin{align*}
\Gamma | H(x) - H(\hat{x}_0) |  & \leq F(x) \\
\mathbb{E} \Big[ \| \nabla f_{\xi} (x, \xi)\|_{\ast}^2 \Big] & \leq  \mathcal{M}  F(x). 
\end{align*}
We notice that in the Euclidean case where $H(x) = \| \cdot \|_2^2$, our setting recovers that described in \cite{nesterov2023randomized}, where
\begin{align*}
\Gamma \| x - x_\star \|_2^2 \leq F(x).
\end{align*}
The first line is an extension of the quadratic lower bound assumption over the objective function and the second line controls the stochastic oracle noise using the objective function. In order to solve the previous problem without prior knowledge on the constants $\Gamma,\mathcal{M}$ or the location of $\hat{x}_0$, we have designed the following composite mirror descent algorithm. 
\begin{algorithm}[H]
\caption{Composite Mirror Descent for Relative scale}
\label{Algo Power SMD}
\begin{algorithmic}
\Require Number of iterations $T \geq 0$, starting point $x_1 \in \mathcal{X}$, step-sizes $(\alpha_t,\beta_t, \gamma_t)_{t\geq 1}$, regularizer $H$ 
\For{$ 1 \leq t \leq T$} 
\begin{equation}
 x_{t+1}  = \arg\min_{x\in \mathcal{X}} \{ \alpha_t  \langle \nabla f_{\xi_t}(x_t) , x \rangle + \beta_t H(x)  +  \gamma_t D^{H}(x,x_{t}) \}
\end{equation}
\EndFor
\State Output $x_{t}^{ag}:=  \frac{\sum_{t=1}^{T}  \alpha_{t} x_{t}}{\sum_{t=1}^{T} \alpha_{t}}$. 
\end{algorithmic}
\end{algorithm}
Unlike in the original setting, we will also suppose that $F$ is explicitly bounded by an a priori unknown constant $\mathcal{L}$. Our convergence result is similar to existing results in \cite{NJLS09,JN14,daspremont2022optimal} with some extra terms due to the new constraints.
\begin{theorem} \label{thm relative power smd}
Assume that $H$ is $1$-strongly-convex with regard to $\| \cdot \|$ and there exist $\Gamma, \mathcal{M}, \mathcal{L}>0$ such that for all $x \in \mathcal{X}$,
\begin{align*}
\Gamma (H(x) - \min H) & \leq F(x) \\
\mathbb{E} \Big[ \| \nabla f_{\xi} (x, \xi)\|_{\ast}^2 \Big] & \leq  \mathcal{M}  F(x) \\
F(x) - F_\star & \leq \mathcal{L}.
\end{align*}
Suppose Algorithm \ref{Algo Power SMD} runs under the oblivious step-size (\ref{oblivious-stepsize}). For any $\mu$ satisfying $ 2\mathcal{M} n /T \leq \mu \leq \frac{\Gamma}{2} $, we have
\begin{equation}
\begin{aligned}
 \mathbb{E}  \left[ F(x_{T}^{ag})\right]  - F_\star   = O  \left(  \Big( \frac{\mathcal{M}n}{\mu T}  + \frac{\mu}{\Gamma} \Big)  F_\star +  \frac{\mu (H(\hat{x}_0) - H_\star) }{T}  + \mathcal{L}\Big( \frac{\mathcal{M}n}{\mu T}\Big)^{n+1} + \frac{\mu D^{H}(x_F,x_{1}) }{T^{n+1}}  \right) 
\end{aligned}
\end{equation}
\end{theorem}
We note that the first two terms with $F_\star$ and $H_\star$ are part of the standard result for optimization in relative scale. A large part of the proof follows the same scheme as \cite{GL12II}, \cite{nesterov2023randomized} and \cite{daspremont2022optimal}.
\begin{proof}
To simplify notation, we write $g_t = \nabla f_{\xi_t} (x_t)$. By the proximal lemma (Lemma \ref{lem:prox}) applied to $u=x$, $y=x_t$,  $u^{\star}=x_{t+1}$,  $f(\cdot) = \alpha_t \langle g_t , \cdot \rangle + \beta_t H(\cdot)  $, and $\nu(\cdot)= \mu \gamma_t  H(\cdot)$, we have for all $x\in \mathcal{X}$
\begin{align*}
 \alpha_t \langle g_t , x_{t+1} - x \rangle + \beta_t H(x_{t+1}) - \beta_t H(x) +  \beta_t D^{H}(x,x_{t+1}) \\
\leq  \gamma_t \Big( D^{H}(x, x_{t}) - D^{H}(x,x_{t+1}) \Big) -  \gamma_t D^{H}( x_{t+1}, x_t).
\end{align*}
We write the stochastic noise by $\Delta_t(x_t) := g_t - \nabla F(x_t)$. By convexity of $F$, we obtain
\begin{equation*}
\begin{aligned}
\langle g_t , x_{t+1} - x \rangle & = \langle \nabla F(x_t) + \Delta_t(x_t) , x_{t} - x \rangle + \langle g_t , x_{t+1} - x_t \rangle  \\ 
& \geq F(x_t) - F(x) + \langle \Delta_t(x_t) , x_{t} - x \rangle + \langle g_t , x_{t+1} - x_t \rangle  \\ 
\end{aligned}
\end{equation*}
Combining the two previous equations, we know for all $x \in \mathcal{X}$
\begin{align*}
 \alpha_t \Big( F(x_t) - F(x) + \langle \Delta_t(x_t) , x_{t} - x \rangle + \langle g_t , x_{t+1} - x_t \rangle  \Big) + \beta_t \Big[  H(x_{t+1}) -  H(x)\Big] +  \beta_t D^{H}(x,x_{t+1}) \\
\leq   \gamma_t \Big( D^{H}(x, x_{t}) - D^{H}(x,x_{t+1}) \Big) -  \gamma_t D^{H}( x_{t+1}, x_t).
\end{align*}
After rearranging the terms, we have 
\begin{align*}
 \alpha_t \Big( F(x_t) - F(x) \Big)  & \leq   \alpha_t \langle g_t , x_t- x_{t+1}  \rangle   -  \gamma_t D^{H}( x_{t+1}, x_t)  \\ 
& + \gamma_t  D^{H}(x, x_{t}) - (\beta_t + \gamma_t) D^{H}(x,x_{t+1}) \\
& +  \alpha_t \langle \Delta_t(x_t) , x- x_{t}  \rangle +   \beta_t \Big[  H(x) - H(x_{t+1}) \Big] 
\end{align*}
We notice that, similarly to the standard stochastic mirror descent proof \cite{daspremont2022optimal}, the first two terms can be simplified by using the strong convexity of $H$ and Young's inequality
\begin{align*}
\alpha_t \langle g_t , x_t - x_{t+1}   \rangle -  \gamma_t D^{H}( x_{t+1}, x_t) & \leq \alpha_t \langle g_t , x_t - x_{t+1}   \rangle - \frac{\gamma_t}{2} \|x_{t+1} - x_t \|^{2} \\
& \leq  \frac{\alpha_t^2}{2 \gamma_t} \| g_t\|_{\ast}^2
\end{align*}
Summing the previous equation from $t=1$ to $t=T$, for all $x \in \mathcal{X}$
\begin{equation}
\label{decomposition-1}
\begin{aligned}
& \sum_{t=1}^{T} \alpha_t [ F(x_{t}) - F(x)  ] 
\leq \gamma_{1} D^{H}(x,x_{1}) -  (\beta_T + \gamma_T) D^{H}(x,x_{T+1}) \\
 & + \sum_{t=1}^{T-1} \underbrace{(\gamma_{t+1} - \gamma_{t} -\beta_t)}_{\leq 0} D^{H}(x,x_{t+1}) 
+  \sum_{t=1}^{T} \alpha_t \langle \Delta_{t}(x_t) , x -x_{t}  \rangle + \sum_{t=1}^{T} \frac{\alpha_t^2 \| g_t \|_{\ast}^{2}}{ \gamma_t} \\
& +  \sum_{t=1}^{T} \beta_t \Big( H(x) - H(x_{t+1}) \Big) .
\end{aligned}
\end{equation}
We decompose the regularizer term $H(x) - H(x_{t+1})$ in the following way 
\begin{align*}
H(x) - H(x_{t+1})  & = \Big( H(\hat{x}_0) - H(x_{t+1}) \Big) +  \Big( H(x) - H(\hat{x}_{0}) \Big) \\
& \leq  \frac{F(x_{t+1}) + F(x) }{\Gamma}  
\end{align*}
where we have used the relative scale setting $ | H(x) - H(\hat{x}_0) | \leq {F(x)}/{\Gamma}$ in the last line. Exceptionally for the iterate $t=T$, we use another decomposition
\[ H(x) - H(x_{T+1}) \leq H(\hat{x}_0) - H(x_{T+1}) + \frac{F(x)}{\Gamma} \leq H(\hat{x}_0) - \min H + \frac{F(x)}{\Gamma} \]
After introducing these decompositions in equation \eqref{decomposition-1}, we simplify and obtain
\begin{equation*}
\begin{aligned}
& \sum_{t=1}^{T} \alpha_t [ F(x_{t}) - F(x)  ]  \leq  \gamma_{1} D^{H}(x,x_{1})  + \frac{F(x)}{\Gamma} \sum_{t=1}^{T} \beta_t +  \sum_{t=1}^{T-1} \frac{\beta_t }{\Gamma} F(x_{t+1})  \\
&    + \beta_T (H(\hat{x}_0) - \min H ) +  \sum_{t=1}^{T} \alpha_t  \langle \Delta_{t}(x_t) , x -x_{t}  \rangle + \sum_{t=1}^{T} \frac{\alpha_t^2 \| g_t \|_{\ast}^{2}}{ \gamma_t} .
\end{aligned}
\end{equation*}
From now on, we fix $x = x_F := \arg\min F $ and consider the expectation
\begin{equation*}
\begin{aligned}
& \sum_{t=1}^{T} \alpha_t \Big(  \mathbb{E}    F(x_{t}) - F_\star  \Big)  \leq  \gamma_{1} D^{H}(x_F,x_{1})  + \frac{2F_\star}{\Gamma} \sum_{t=1}^{T} \beta_t  \\
& + \beta_T \Big(H(\hat{x}_0) - \min H - \frac{F_\star}{\Gamma} \Big)  +  \sum_{t=1}^{T-1} \frac{\beta_t}{\Gamma} \Big(  \mathbb{E}    F(x_{t+1}) - F_\star  \Big)+  \sum_{t=1}^{T} \frac{\alpha_t^2 \mathbb{E} \| g_t \|_{\ast}^{2}}{ \gamma_t} .
\end{aligned}
\end{equation*}
We also notice that $ H(\hat{x}_0) - \min H  \leq \frac{F(x_H)}{\Gamma}$ if we write $x_H = \arg \min_{x \in \mathcal{X}} H(x)$. We rearrange the previous equation in the following way
\begin{equation}
\label{relative-scale-step-1}
\begin{aligned}
 & \sum_{t=1}^{T} \frac{\alpha_t}{2}  \Big(  \mathbb{E}    F(x_{t}) - F_\star  \Big) \\
 \leq & \gamma_{1} D^{H}(x_F,x_{1}) +  \frac{2F_\star}{\Gamma}  \sum_{t=1}^{T} \beta_t +  \sum_{t=1}^{T} \frac{\alpha_t^2 \mathbb{E} \| g_t \|_{\ast}^{2}}{ \gamma_t} + \beta_T \Big(H(\hat{x}_0) - \min H  -\frac{ F_\star}{\Gamma} \Big) \\
 + & \sum_{t=1}^{T-1} \frac{\beta_{t-1}}{\Gamma}  \Big(  \mathbb{E}    F(x_{t}) - F_\star  \Big) - \sum_{t=1}^{T} \frac{\alpha_t}{2}  \Big(  \mathbb{E}    F(x_{t}) - F_\star  \Big) \\
 \leq &  \gamma_{1}  D^{H}(x_F,x_{1}) +  \frac{2F_\star}{\Gamma}  \sum_{t=1}^{T} \beta_t   +  \sum_{t=1}^{T} \frac{\alpha_t^2}{ \gamma_t} \Big( \mathcal{M} \mathbb{E} F(x_t) \Big) + \beta_T \Big(H(\hat{x}_0) - \min H  -\frac{ F_\star}{\Gamma} \Big)  \\ 
 + & \sum_{t=1}^{T-1} \frac{\beta_{t-1}}{\Gamma}  \Big(  \mathbb{E}    F(x_{t}) - F_\star  \Big) - \sum_{t=1}^{T} \frac{\alpha_t}{2}  \Big(  \mathbb{E}    F(x_{t}) - F_\star  \Big) \\
 \leq & \gamma_{1}  D^{H}(x_F,x_{1}) +  \frac{2F_\star}{\Gamma}  \sum_{t=1}^{T} \beta_t  + \sum_{t=1}^{T} \frac{\alpha_t^2}{ \gamma_t} \Big( \mathcal{M}  F_\star \Big) + \beta_T \Big(H(\hat{x}_0) - \min H  -\frac{ F_\star}{\Gamma} \Big)   \\
 & +   \sum_{t=1}^{T} \Big( \frac{\mathcal{M}\alpha_t^2}{ \gamma_t} - \frac{\alpha_t}{4} \Big) \Big(  \mathbb{E}F(x_t) - F_\star \Big) + \sum_{t=1}^{T-1} \Big( \frac{\beta_{t-1}}{\Gamma} - \frac{\alpha_t}{4} \Big) \Big(  \mathbb{E}    F(x_{t}) - F_\star  \Big)
\end{aligned}
\end{equation}
Now for any $n \geq 0$, if we consider the polynomial step-size $\alpha_t \propto t^n, \beta_t \propto \mu t^n, \gamma_{t} \propto \mu t^{n+1}/n $, we notice the two terms in the last line will become negative when $t$ grows. More precisely, we note that
\begin{align*}
\frac{\mathcal{M} \alpha_t^2 }{ \gamma_t}  \leq \frac{\alpha_t}{4} & \Leftrightarrow 4\mathcal{M}  \leq \frac{\gamma_t}{ \alpha_t} \propto \frac{\mu t}{n}  \\
\frac{\alpha_t}{4} \geq \frac{\beta_{t-1}}{\Gamma}& \Leftrightarrow t^n \geq \frac{4\mu}{\Gamma} (t-1)^n
\end{align*}
Since we supposed that the choice of $\mu$ satisfies the regime $ 2\mathcal{M}n /T \leq \mu \leq \frac{\Gamma}{2} $, we know there exists a transition time $T_0 \propto \lceil  \mathcal{M}n /\mu \rceil $ where the last two terms become negative. Now if we introduce the transition time in \eqref{relative-scale-step-1}, we obtain
\begin{align*}
 \mathbb{E} \left[ \sum_{t=1}^{T} t^n \Big( F(x_{t}) - F_\star  \Big) \right]  & \lesssim  \mu   \left( \frac{1}{n} D^{H}(x_F,x_{1})  +    \frac{F_\star}{\Gamma}  \sum_{t=1}^{T-1} t^n   \right) +   \sum_{t=1}^{T} \frac{  \mathcal{M} n t^{2n}}{ \mu t^{n+1}}  F_\star  
 \\
  &  + \mu \left( H(\hat{x}_0) - H_\star - \frac{F_\star}{\Gamma} \right)T^n + \mathcal{L} \sum_{t=1}^{T_0} \frac{  \mathcal{M} n t^{2n}}{ \mu t^{n+1}}  
\end{align*}
We note that 
\[ \sum_{t=1}^{T_0} \frac{  \mathcal{M} n t^{2n}}{ \mu t^{n+1}}  \propto \frac{\mathcal{M} n}{\mu} \sum_{t=1}^{T_0} t^{n-1} \propto \frac{\mathcal{M} }{\mu}  T_0^{n} \]
We conclude by using the convexity of $F$ with $x_T^{ag} =\frac{\sum \alpha_t x_t }{\sum \alpha_t}$
\begin{align*}
 \mathbb{E}  \left[ F(x_{T}^{ag})\right]  - F_\star   = O  \left(  \Big( \frac{\mathcal{M}n}{\mu T} + \frac{\mu}{\Gamma} \Big)  F_\star + \frac{\mu (H(\hat{x}_0) - H_\star) }{T}   + \mathcal{L}\Big( \frac{\mathcal{M}n}{\mu T}\Big)^{n+1} + \frac{\mu D^{H}(x_F,x_{1}) }{T^{n+1}}  \right)  
\end{align*}
\end{proof}
Note that if $\hat{x}_0$ is known, the second term disappears by choosing $H$ centered at $\hat{x}_0$. We illustrate the previous result with concrete choices of $(n,\mu)$ in the next corollary.
\begin{corollary}
Under the previous assumption of relative scale setting, we choose $(n,\mu) = \Big(3,\frac{1}{\sqrt{T}} \Big)$.  
If $T$ is large enough i.e. $T \geq \max \Big( 36 \mathcal{M}^2, \frac{4}{\Gamma^2} \Big)$ then our algorithm yields,
\begin{equation}
\begin{aligned}
 \mathbb{E}  \left[ F(x_{T}^{ag})\right]  - F_\star   = O  \left(  \Big( \mathcal{M} + \frac{1}{\Gamma} \Big) \frac{F_\star}{\sqrt{T}} + \frac{H(\hat{x}_0) - H_\star}{\sqrt{T}}   + \frac{\mathcal{L} \mathcal{M}^4}{T^2} + \frac{\mu D^{H}(x_F,x_{1}) }{T^4}  \right)   
\end{aligned}
\end{equation}
\end{corollary}
\begin{proof}
We apply the previous Theorem and verify that for $\mu = \frac{1}{\sqrt{T}}$ and $n=3$, 
\[ 2\mathcal{M}n /T \leq \mu \leq \frac{\Gamma}{2} \quad \Leftrightarrow \frac{6\mathcal{M}}{T} \leq \frac{1}{\sqrt{T}} \leq \frac{\Gamma}{2} \quad \Leftrightarrow \frac{36\mathcal{M}}{T} \leq 1 \leq \frac{\Gamma^2 T}{4} \]
where we multiply by $T$ in the last step.
\end{proof}

\subsection{Accelerated Composite Mirror descent}

For the smooth setting, we use the acceleration framework from \cite{daspremont2022optimal} with step-size \eqref{oblivious-stepsize}.
\begin{algorithm}
\caption{Accelerated Composite Mirror Descent}
\label{Algo ACSMD}
\begin{algorithmic}
\Require Number of iterations $T \geq 0$, starting point $x_1 \in \mathcal{X}$, step-sizes $(\alpha_t,\beta_t, \gamma_t)_{t\geq 1}$, regularizer $H$ 
\For{$ 1 \leq t \leq T$}
\State Compute $x_{t}^{md} = \frac{A_{t-1}}{A_{t}} x_{t}^{ag} + \frac{\alpha_{t}}{A_{t}} x_{t}$

\State Compute $x_{t+1} = \arg\min_{x\in \mathcal{X}} \{\alpha_t \langle \nabla f_{\xi_t}(x_{t}^{md}) , x \rangle +  \beta_t H(x) + \gamma_t D^{H}(x,x_{t}) \}$

\State Compute $x_{t+1}^{ag} = \frac{A_{t-1}}{A_{t}} x_{t}^{ag} + \frac{\alpha_{t}}{A_{t}} x_{t+1}$

\EndFor
\State Output $x_{t+1}^{ag}$. 
\end{algorithmic}
\end{algorithm}

As in the relative case, we show the convergence of the accelerated algorithm, by introducing a burn in time $T_0$.

\begin{theorem} \label{thm smooth}
Assume that $H$ is $1$-strongly-convex with regard to $\| \cdot \|$ and that there exist $D, D_H, L, \sigma >0$ such that for all $x,y \in \mathcal{X}$,
\begin{align*}
D^{F}(x,y) & \leq \frac{L}{2} \|x-y\|^2 \\
\mathbb{E}_{\xi} \Big[ \| \nabla f_{\xi}(x) - \nabla F(x) \|_{\ast}^2 \Big] & \leq \sigma^2 \\
\| x- y\|^2 & \leq D^2 \\
H(x_F) - H(x) & \leq D^2
\end{align*}
where $x_F$ is the minimum of $F$. Suppose Algorithm \ref{Algo ACSMD} runs under the oblivious step-size~(\ref{oblivious-stepsize}). For any $\mu >0$, if $ 2 \sqrt{L/\mu} \leq T $, then
\begin{equation}
\begin{aligned}
 \mathbb{E} [F(x_{T+1}^{ag})]- F_\star  =O \left( \frac{\mu  D^{H}(x_F,x_{1})}{nT^{n+1}} + \mu D^2 +  \frac{n \sigma^2}{\mu T }+   \frac{LD^2}{T} \Big( \frac{Ln}{\mu T^2}\Big)^{n/2} \right)  
\end{aligned}
\end{equation}
\end{theorem}
\begin{proof}
Applying the proximal lemma to $u=x$, $y=x_t$, $u^{\star}=x_{t+1}$, $f(\cdot) = \alpha_t  \langle \nabla f_{\xi_t}(x_t^{md}),\cdot \rangle + \beta_t H(\cdot)  $, and $\nu(\cdot)=  \gamma_t H(\cdot)$, we have for all $x \in \mathcal{X}$
\begin{align*}
& \alpha_t \langle \nabla f_{\xi_t}(x_t^{md}), x_{t+1} \rangle + \beta_t H(x_{t+1})  +  \gamma_t D^{H}( x_{t+1}, x_t) +  \beta_t D^{H}(x,x_{t+1}) \\
\leq & \alpha_t \langle \nabla f_{\xi_t}(x_t^{md}),x \rangle + \beta_t H(x)  +  \gamma_t [D^{H}(x, x_{t}) - D^{H}(x,x_{t+1})] 
\end{align*}
After rearranging, we have
\begin{equation}
\label{decomposition-2}
\begin{aligned}
& \alpha_t \langle \nabla F(x_t^{md}), x_{t+1} - x  \rangle + \alpha_t \langle \Delta(x_{t}^{md}) , x_{t+1} - x  \rangle   + \beta_t H(x_{t+1}) - \beta_t H(x)  \\
= &  \alpha_t \langle \nabla f_{\xi_t}(x_t^{md}), x_{t+1} - x  \rangle + \beta_t H(x_{t+1}) - \beta_t H(x)    \\
\leq &   \gamma_t D^{H}(x, x_{t}) -  (\beta_t + \gamma_t) D^{H}(x,x_{t+1}) - \gamma_t D^{H}( x_{t+1}, x_t) 
\end{aligned}
\end{equation}
where the stochastic noise is defined as $ \Delta(x_{t}^{md}) = \nabla f_{\xi_t}(x_t^{md}) - \nabla F(x_t^{md}) $. Now we split the gradient term with $\nabla F(x_t^{md})$ in the following way
\begin{align*}
 & \alpha_t \langle \nabla F(x_t^{md}), x_{t+1} - x  \rangle \\
 = &  \alpha_t \langle \nabla F(x_t^{md}), x_{t+1} - x_t^{md}  \rangle + \alpha_t \langle \nabla F(x_t^{md}), x_t^{md} - x  \rangle \\
 = &  \alpha_t \langle \nabla F(x_t^{md}), x_{t+1} - x_t^{md}  \rangle + \alpha_t \Big( F(x_t^{md}) - F(x) + D^F(x,x_t^{md}) \Big)
\end{align*}
The Bregman divergence $ D^F(x,x_t^{md})$ is positive because $F$ is convex, therefore we can reformulate \eqref{decomposition-2} in the following way
\begin{equation*}
\begin{aligned}
& \alpha_t \Big( F(x_t^{md}) - F(x) +   \langle \nabla F(x_t^{md}), x_{t+1} - x_t^{md}  \rangle  \Big) + \beta_t H(x_{t+1}) - \beta_t H(x) \\
\leq & \alpha_t \langle \Delta(x_{t}^{md}) , x- x_{t+1}  \rangle   +  \gamma_t D^{H}(x, x_{t}) -  (\beta_t + \gamma_t) D^{H}(x,x_{t+1}) - \gamma_t D^{H}( x_{t+1}, x_t) 
\end{aligned}
\end{equation*}
For the acceleration step, we recall that $x_{t+1}^{ag} = \frac{A_{t-1}}{A_t} x_t^{ag} + \frac{\alpha_t}{A_t} x_{t+1}$. From the smoothness of $F$, we know 
\begin{align*}
F(x_{t+1}^{ag}) & \leq F(x_{t}^{md}) + \langle \nabla F(x_{t}^{md}) , x_{t+1}^{ag} -x_{t}^{md} \rangle + \frac{L}{2} \| x_{t+1}^{ag} -x_{t}^{md}  \|^2 \\
&= \textstyle F(x_{t}^{md}) + \langle \nabla F(x_{t}^{md}) ,\frac{A_{t-1}}{A_{t}} x_{t}^{ag} + \frac{\alpha_{t}}{A_{t}} x_{t+1} -x_{t}^{md} \rangle + \frac{L}{2} \frac{\alpha_{t}^2}{A_{t}^2} \| x_{t+1} -x_t  \|^2  \\
& = \frac{A_{t-1}}{A_{t}} \left[F(x_{t}^{md}) +   \langle \nabla F(x_{t}^{md}) , x_{t}^{ag} -x_{t}^{md}  \rangle \right]  \\
& + \frac{\alpha_{t}}{A_{t}}  \left[ F(x_{t}^{md}) + \langle \nabla F(x_{t}^{md}) ,  x_{t+1} -x_{t}^{md}\rangle \right]   + \frac{L}{2} \frac{\alpha_{t}^2}{A_{t}^2} \| x_{t+1} -x_t  \|^2, 
\end{align*}
by scaling every term by $A_t$, we know that for all $x \in \mathcal{X}$,
\begin{align*}
& A_t \Big( F(x_{t+1}^{ag}) - F(x) \Big) \\
\leq & A_{t-1} \Big( F(x_t^{md}) - F(x)  +  \langle \nabla F(x_{t}^{md}) , x_{t}^{ag} -x_{t}^{md}  \rangle \Big) \\
+ & \alpha_t \Big( F(x_t^{md}) - F(x) +   \langle \nabla F(x_{t}^{md}) , x_{t+1} -x_{t}^{md}  \rangle \Big) + \frac{L}{2} \frac{\alpha_{t}^2}{A_{t}} \| x_{t+1} -x_t  \|^2
\end{align*}
Combining with previous equations, we have
\begin{equation}
\begin{aligned}
& A_t \Big( F(x_{t+1}^{ag}) - F(x) \Big) + \beta_t \Big( H(x_{t+1}) - H(x) \Big) \\
\leq & A_{t-1} \Big( F(x_t^{md}) - F(x)  +  \langle \nabla F(x_{t}^{md}) , x_{t}^{ag} -x_{t}^{md}  \rangle \Big) + \frac{L}{2} \frac{\alpha_{t}^2}{A_{t}} \| x_{t+1} -x_t  \|^2 \\
+ & \alpha_t \langle \Delta(x_{t}^{md}) , x- x_{t+1}  \rangle   +  \gamma_t D^{H}(x, x_{t}) -  (\beta_t + \gamma_t) D^{H}(x,x_{t+1}) - \gamma_t D^{H}( x_{t+1}, x_t) 
\end{aligned}
\end{equation}
We also notice that by the convexity of $F$
\begin{align*}
\langle \nabla F(x_{t}^{md}) , x_{t}^{ag} -x_{t}^{md}  \rangle \leq F(x_t^{ag}) - F(x_t^{md})
\end{align*}
which implies 
\begin{equation}
\begin{aligned}
& A_t \Big( F(x_{t+1}^{ag}) - F(x) \Big) + \beta_t \Big( H(x_{t+1}) - H(x) \Big) \\
\leq & A_{t-1} \Big( F(x_t^{ag}) - F(x)  \Big) + \frac{L}{2} \frac{\alpha_{t}^2}{A_{t}} \| x_{t+1} -x_t  \|^2 + \alpha_t \langle \Delta(x_{t}^{md}) , x- x_{t}  \rangle \\
+ & \alpha_t \langle \Delta(x_{t}^{md}) , x_t - x_{t+1}  \rangle   +  \gamma_t D^{H}(x, x_{t}) -  (\beta_t + \gamma_t) D^{H}(x,x_{t+1}) - \gamma_t D^{H}( x_{t+1}, x_t).
\end{aligned}
\end{equation}
We use the strong convexity of $H$ and split the term into two parts
\[ \gamma_t D^{H}( x_{t+1}, x_t) \geq \frac{\gamma_t}{4} \| x_{t+1} - x_t\|^2 + \frac{\gamma_t}{4} \| x_{t+1} - x_t\|^2, \]
and we apply Young's inequality 
\[ \alpha_t \langle \Delta(x_{t}^{md}) , x_{t} -x_{t+1}  \rangle - \frac{\gamma_t}{4} \| x_{t+1} - x_t\|^2 \leq \frac{2\alpha_t^2}{  \gamma_t} \| \Delta(x_{t}^{md})\|_\ast^2. \]
We use the other $\gamma_t/4$ to balance the $L\alpha_t^2/2A_t$ term. In the end we obtain for all $x\in \mathcal{X}$
\begin{equation}
\begin{aligned}
& A_t \Big( F(x_{t+1}^{ag}) - F(x) \Big) - A_{t-1} \Big( F(x_t^{ag}) - F(x)  \Big)  \\
\leq & \beta_t \Big( H(x_{t+1}) - H(x) \Big) +  \Big( \frac{L}{2} \frac{\alpha_{t}^2}{A_{t}} - \frac{\gamma_t}{4} \Big) \| x_{t+1} -x_t  \|^2 + \alpha_t \langle \Delta(x_{t}^{md}) , x- x_{t}  \rangle \\
+ &   \gamma_t D^{H}(x, x_{t}) -  (\beta_t + \gamma_t) D^{H}(x,x_{t+1})  + \frac{2\alpha_t^2}{\gamma_t} \| \Delta(x_{t}^{md})\|_\ast^2 
\end{aligned}
\end{equation}
Adding all terms from $t=1$ to $t=T$, we obtain
\begin{align*}
&  A_T \Big( \left[F(x_{t+1}^{ag}) \right] - F(x) \Big) +  \gamma_{T} D^{H}(x, x_{T+1})  \\
\leq &  \gamma_1 D^{H}(x,x_{1}) + \sum_{t=1}^{T} \alpha_t \langle \Delta(x_{t}^{md}) , x -x_{t}  \rangle +  \sum_{t=1}^{T} \beta_t \Big( H(x)- H(x_{t+1}) \Big) \\
+ & \sum_{t=1}^{T} \frac{2\alpha_t^2}{\gamma_t} \| \Delta(x_{t}^{md})\|_\ast^2 + \sum_{t=1}^{T} \Big( \frac{L}{2} \frac{\alpha_{t}^2}{A_t} - \frac{\gamma_t}{4}\Big) \| x_{t+1} - x_t\|^2  \\
+ & \sum_{t=1}^{T-1} \left( \gamma_{t+1} - \gamma_{t}-\beta_t \right) D^{H}(x,x_{t+1}).
\end{align*}
We know that the conditional expectation on $t$ for $ \mathbb{E} [ \alpha_t \langle \Delta(x_{t}^{md}) , x -x_{t}  \rangle ] = 0$. By the tower law, it disappears when we consider the sum of the expectation. By fixing 
 $x = x_F := \arg\min F $, we obtain
\begin{equation}
\begin{aligned}
& A_T \Big( \mathbb{E} \left[F(x_{t+1}^{ag})\right]  - F_\star \Big) \\
\leq &  \gamma_1 D^{H}(x_F,x_{1}) +  \sum_{t=1}^{T} \beta_t \mathbb{E} \Big( H(x_F)- H(x_{t+1}) \Big) + \sum_{t=1}^{T} \frac{2 \sigma^2 \alpha_t^2}{\gamma_t}  \\
+ &  \sum_{t=1}^{T} \Big( \frac{L}{2} \frac{\alpha_{t}^2}{A_t} - \frac{\gamma_t}{4}\Big) \| x_{t+1} - x_t\|^2  + \sum_{t=1}^{T-1} \left( \gamma_{t+1} - \gamma_{t}-\beta_t \right) D^{H}(x_F,x_{t+1}).
\end{aligned}
\end{equation}
Recall that for all $x,y \in \mathcal{X}$, $H(x_{F}) - H(x), \| x-y\| \leq D^2$. Now for any $n \geq 0$, if we consider the polynomial step-size $\alpha_t \propto t^n, \beta_t \propto \mu t^n, \gamma_{t} \propto \mu t^{n+1}/n $, we look at the transition time $T_0 :=  \arg\max_{t} \{ 2L \leq \frac{\gamma_t A_t}{\alpha_{t}^2} \}$. For $\mu > 0$,
\begin{align*}
2L \lesssim \frac{\gamma_t A_t}{\alpha_t^2} 
\Leftrightarrow \quad  2L \lesssim \frac{\mu t^{2n+2}}{n^2 t^{2n}}  \Leftrightarrow \quad \frac{\sqrt{2L}n}{\sqrt{\mu}}  \lesssim t  
\end{align*}
Therefore $T_0 \propto n \sqrt{\frac{L}{\mu}}  $ and we have,
\begin{equation}
\begin{aligned}
& A_T \Big( \mathbb{E}  \left[F(x_{t+1}^{ag})\right] - F_\star \Big) \\
\leq &  \gamma_1 D^{H}(x_F,x_{1})  + D^2 \sum_{t=1}^{T} \beta_t + D^2 \sum_{t=1}^{T_0} \Big( \frac{L}{2} \frac{\alpha_{t}^2}{A_t} - \frac{\gamma_t}{4}\Big)    \\
+ &  \sum_{t=1}^{T} \frac{2 \sigma^2 \alpha_t^2}{\gamma_t} + \sum_{t=1}^{T-1} \left( \gamma_{t+1} - \gamma_{t}-\beta_t \right) D^{H}(x_F,x_{t+1}).
\end{aligned}
\end{equation}
The remaining step is to simplify the transition time term. We note that
\[ \sum_{t=1}^{T_0} \frac{\alpha_t^2}{A_t} \propto \sum_{t=1}^{T_0} \frac{n t^{2n}}{t^{n+1}} \propto n \sum_{t=1}^{T_0} t^{n-1}   \propto T_0^n \]
By diving every term by $A_T$, we conclude that
\begin{equation}
\begin{aligned}
\mathbb{E} [F(x_{T+1}^{ag})]- F_\star  \lesssim \frac{\mu  D^{H}(x_F,x_{1})}{n T^{n+1}} + \mu D^2 +  \frac{n\sigma^2}{\mu T }+   \frac{LD^2}{T} \Big( \frac{Ln}{\mu T^2}\Big)^{n/2}  
\end{aligned}
\end{equation}

\end{proof}

We would like to underline that transition time analysis can co-exist with the original analysis \cite{daspremont2022optimal}. We will illustrate the previous result by choosing $(n,\mu)$.

\begin{corollary}
Under the assumptions of Theorem~\ref{thm smooth},  if we tune $(n,\mu) = \Big( 2, \frac{1}{\sqrt{T}} \Big)$ (with $ 3 L^{2/3} \leq T $), then
\begin{equation}
\begin{aligned}
 \mathbb{E} [F(x_{T+1}^{ag})]- F_\star  = O  \left( \frac{  D^{H}(x_F,x_{1})}{T^{3.5}}  +  \frac{\sigma^2 + D^2}{\sqrt{T} }+   \frac{LD^2}{T} \Big( \frac{L}{ T^{3/2}}\Big)  \right)
\end{aligned}
\end{equation}
\end{corollary}

\begin{proof}
We apply the previous Theorem with $n=2$ and $\mu = \frac{1}{\sqrt{T}}$, we note that 
\begin{align*}
 2 \sqrt{L/\mu} \leq T \quad \Leftrightarrow 4 L \sqrt{T} \leq T^2 \Leftrightarrow 4 L \leq T^{3/2} 
\end{align*}
\end{proof}

\subsection{Composite Mirror descent for Relative smoothness}
For the relative smoothness setting, we recall for each random sampling $\xi$, $f_\xi$ is $L$ relatively smooth with respect to $H$ if for all $x,y\in \mathcal{X}, $ \[ D^{f_\xi}(x,y) \leq L_\star D^H(x,y) \]

The concept of randomness with relative smoothness is more subtle than the usual way since the standard variance could be arbitrary large in the relative smooth case. We will consider the new variance definition from \cite{Dragomir2021FastSB}, we suppose there exists $\eta_0 $ and $\sigma^2$ such that for all $x \in \mathcal{X}$ and $\eta < \eta_0$
\begin{equation}
\label{variance-realtive}
\mathbb{E}_\xi \left[ D^{H_\star}\Big( \nabla H(x) - 2\eta \nabla f_{\xi}(x_F), \nabla H(x)\Big) \right] \leq 2\eta^2 \sigma^2. 
\end{equation}
For our composite mirror descent, we will also suppose that for all $\eta > 0$
\begin{equation}
\label{bounded-realtive}
\mathbb{E}_\xi \left[ D^{H_\star}\Big( \nabla H(x) - 2\eta \nabla f_{\xi}(x_F), \nabla H(x)\Big) \right] \leq 2 \eta^2 G^2
\end{equation}
The constant $G$ could be much larger than $\sigma$, but we will see that its impact is limited in our convergence result.

\begin{algorithm}[H]
\caption{Composite Mirror descent for Relative smoothness }
\label{Algo Relative smoothness}
\begin{algorithmic}
\Require Number of iterations $T \geq 0$, starting point $x_1 \in \mathcal{X}$, step-sizes $(\alpha_t,\beta_t, \gamma_t)_{t\geq 1}$, regularizer $H$ 
\For{$ 1 \leq t \leq T$} 
\begin{equation}
 x_{t+1}  = \arg\min_{x\in \mathcal{X}} \{ \alpha_t  \langle \nabla f_\xi(x_t) , x \rangle + \beta_t H(x)  + \gamma_t D^{H}(x,x_{t}) \}
\end{equation}
\EndFor
\State Output $x_{t}^{ag}:=  \frac{\sum_{t=1}^{T}  \alpha_{t} x_{t}}{\sum_{t=1}^{T} \alpha_{t}}$. 
\end{algorithmic}
\end{algorithm}

\begin{theorem} 
We recall that for the bounded relative scale case, we assume that there exists $\eta_0, \sigma, G, L_\star >0$ such that for all $x,y \in \mathcal{X}$,
\begin{align*}
\forall \xi, \quad D^{f_\xi}(x,y) & \leq L_\star D^H(x,y) \\
\forall \eta< \eta_0, \quad \mathbb{E}_\xi \left[ D^{H_\star}\Big( \nabla H(x) - 2\eta \nabla f_{\xi}(x_F), \nabla H(x)\Big) \right] & \leq 2\eta^2 \sigma^2 \\
\forall \eta, \quad \mathbb{E}_\xi \left[ D^{H_\star}\Big( \nabla H(x) - 2\eta \nabla f_{\xi}(x_F), \nabla H(x)\Big) \right] & \leq 2 \eta^2 G^2 \\
\exists \hat{x}_0 \in \mathcal{X}, 
\nabla H(\hat{x}_0) & = 0
\end{align*}
Suppose Algorithm \ref{Algo Power SMD} runs under the oblivious step-size (\ref{oblivious-stepsize}).For any $\mu > \max \Big( \frac{L_\star}{T}, \frac{\eta_0}{T} \Big)$ , 
\begin{equation}
\begin{aligned}
   \frac{\sum_{t=1}^{T} n t^n \mathbb{E}[ D^{F}(x_\star,x_t)]}{T^{n+1}}   = \mathcal{O}  \left(  \frac{\mu D^{H}(x_F,x_1)}{n T^{n+1}} + \mu D^2 + \frac{\sigma^2 n }{\mu T } + \frac{G^2}{n \mu T^{n+1}}  \max \Big\{ \frac{L^n}{\mu^n}, \frac{\eta_0^n}{\mu^n }   \Big\} \right) 
\end{aligned}
\end{equation}
\end{theorem}

\begin{proof}
Since we supposed that the mirror descent steps stay in the interior of the domain $\mathcal{X}$, we know that the gradients at iteration $t$ is zero at $x_{t+1}$
\begin{equation}
\label{eq-relative-smooth-step1}
\begin{aligned}
x_{t+1} = \arg\min_{x\in \mathcal{X}} \{ \alpha_t  \langle \nabla f_\xi(x_t) , x \rangle + \beta_t H(x)  + \gamma_t D^{H}(x,x_{t}) \} \\
\implies  \alpha_t   \nabla f_\xi(x_t) + \beta_t \nabla  H(x_{t+1})  + \gamma_t \Big(\nabla H (x_{t+1}) - \nabla H(x_t) \Big) = 0 \\
\Leftrightarrow \alpha_t  \nabla f_{\xi}(x_t)  =  \gamma_t \Big[ \nabla H(x_t) - \nabla H(x_{t+1}) \Big] - \beta_t \nabla H(x_{t+1}) \\
\Leftrightarrow (\beta_t + \gamma_t) \nabla H(x_{t+1})  =  \gamma_t  \nabla H(x_t) - \alpha_t  \nabla f_{\xi}(x_t) 
\end{aligned}
\end{equation}
Now, we apply the proximal lemma \ref{lem:prox-free} where $u=x$, $y=x_t$, $u^{\star}=x_{t+1}$, $f(\cdot) = \alpha_t \langle \nabla f_{\xi}(x_t), \cdot \rangle + \beta_t H(\cdot) $, and $\nu(\cdot)= \gamma_t H(\cdot)$. We obtain for all $x \in \mathcal{X}$
\begin{align*}
& \alpha_t  \langle \nabla f_{\xi}(x_t), x_{t+1} - x \rangle + \beta_t \Big(  H(x_{t+1})  - H(x)   \Big) + \gamma_t D^{H}( x_{t+1}, x_t) +  \beta_t D^{H}(x,x_{t+1}) \\
= &  \gamma_t [D^{H}(x, x_{t}) - D^{H}(x,x_{t+1})] 
\end{align*}
Considering $x= x_\star = x_F $ leads us to 
\begin{equation}
\begin{aligned}
& \alpha_t  \Big( \langle \nabla f_{\xi}(x_t), x_{t+1} - x_\star \rangle \Big) +   \gamma_t D^{H}( x_{t+1}, x_t)  \\
= &  \gamma_t D^{H}(x, x_{t}) -  (\beta_t + \gamma_t) D^{H}(x,x_{t+1}) + \beta_t \Big( H(x_\star) -  H(x_{t+1})   \Big)
\end{aligned}
\end{equation}
After arranging different terms, we have
\begin{equation}
\label{eq-relative-smooth-step2}
\begin{aligned}
& \alpha_t   \langle \nabla f_{\xi}(x_t) , x_{t} - x_\star \rangle  +  \alpha_t   \langle \nabla f_{\xi}(x_t) , x_{t+1} - x_t \rangle +  \gamma_t D^{H}( x_{t+1}, x_t)  \\
= &   \gamma_t D^{H}(x_\star, x_{t}) - (\beta_t + \gamma_t) D^{H}(x_\star,x_{t+1}) + \beta_t \Big( H(x_\star) -  H(x_{t+1})   \Big)
\end{aligned}
\end{equation}
We replace the term $\alpha_t  \nabla f_{\xi}(x_t)$ in \eqref{eq-relative-smooth-step2} with the one in \eqref{eq-relative-smooth-step1}. We obtain
\begin{align*}
& \alpha_t  \langle \nabla f_{\xi}(x_t) , x_{t} - x_\star \rangle  +   \gamma_t   \langle \nabla H(x_t) - \nabla H(x_{t+1}) , x_{t+1} - x_t \rangle  +  \gamma_t D^{H}( x_{t+1}, x_t)  \\
= &   \gamma_t D^{H}(x_\star, x_{t}) -  (\beta_t + \gamma_t) D^{H}(x_\star,x_{t+1})  \\
+ &    \beta_t \Big( H(x_\star) -  H(x_{t+1})   \Big) + \beta_t \langle \nabla H(x_{t+1}), x_{t+1} - x_t \rangle
\end{align*}
We introduce the following simplifications,
\begin{equation}
\begin{aligned}
\langle \nabla H(x_{t+1})  - \nabla H(x_t) , x_{t+1} - x_t \rangle & = D^H(x_t,x_{t+1}) + D^{H}(x_{t+1},x_t) \\
\langle \nabla H(x_{t+1}), x_{t+1} - x_t \rangle & = H(x_{t+1}) - H(x_t)   +  D^H(x_t,x_{t+1}) 
\end{aligned}
\end{equation}
Therefore, we obtain
\begin{equation}
\begin{aligned}
& \alpha_t  \langle \nabla f_{\xi}(x_t) , x_{t} - x_\star \rangle  -   \gamma_t  \Big( D^H(x_t,x_{t+1}) + D^{H}(x_{t+1},x_t)  \Big)   +  \gamma_t D^{H}( x_{t+1}, x_t)  \\
= &   \gamma_t D^{H}(x_\star, x_{t}) -  (\beta_t + \gamma_t) D^{H}(x_\star,x_{t+1}) +     \beta_t \Big( H(x_\star) -  H(x_{t})   \Big) + \beta_t D^H(x_t,x_{t+1}) 
\end{aligned}
\end{equation}
which can be rearranged into
\begin{equation}
\begin{aligned}
 \alpha_t  \langle \nabla f_{\xi}(x_t) , x_{t} - x_\star \rangle &=    \gamma_t D^{H}(x_\star, x_{t}) -  (\beta_t + \gamma_t) D^{H}(x_\star,x_{t+1})  \\
 & +     \beta_t \Big( H(x_\star) -  H(x_{t})   \Big) + (\beta_t + \gamma_t ) D^H(x_t,x_{t+1}) 
\end{aligned}
\end{equation}
Since $\nabla F(x_\star) = 0 $,  by considering the conditional expectation at iteration $t$, we obtain
\begin{align*}
& \alpha_t  \Big( D^{F}(x_t,x_\star) + D^{F}(x_\star,x_t)  \Big) \\
= &   \alpha_t   \langle \nabla F(x_t) - \nabla F(x_\star), x_{t} - x_\star \rangle   \\
= &   \mathbb{E}_{\xi} \Big[ \alpha_t   \langle \nabla f_{\xi}(x_t) , x_{t} - x_\star \rangle  \Big] \\
= &  \gamma_t D^{H}(x_\star, x_{t}) -  (\beta_t + \gamma_t) \mathbb{E} D^{H}(x_\star,x_{t+1})     \\
+ &   \beta_t \Big( H(x_\star) -  H(x_{t})   \Big)  +  (\beta_t + \gamma_t) \mathbb{E}  D^H(x_t,x_{t+1})
\end{align*}
Our goal is to rewrite $\mathbb{E}  D^H(x_t,x_{t+1})$. We use the decomposition of $\nabla H(x_{t+1})$ in \eqref{eq-relative-smooth-step1}
\begin{align*}
 D^H(x_t, x_{t+1})  & = D^{H_\star}(\nabla H(x_{t+1}), \nabla H(x_t) ) \\
& = D^{H_\star}\Big(\frac{\gamma_t}{\beta_t + \gamma_t} \nabla H(x_t) -  \frac{\alpha_t}{\beta_t + \gamma_t} \nabla f_{\xi}(x_t) ,\nabla H(x_t) \Big)  \\
& = D^{H_\star}\left(\frac{\gamma_t}{\beta_t + \gamma_t} \Big( \nabla H(x_t) - \frac{\alpha_t}{\gamma_t} \nabla f_{\xi}(x_t) \Big) + \frac{\beta_t}{\beta_t + \gamma_t} \nabla H(\hat{x}_0), \nabla H(x_t) \right)  
\end{align*}
where $\hat{x}_0 \in \mathcal{X}$ satisfies $\nabla H(\hat{x}_0)=0$. Since $D^{H_\star}(\cdot,\nabla H(x_t))$ is a convex function,
\begin{align*}
 D^H(x_t, x_{t+1})   & \leq \frac{\gamma_t}{\beta_t+ \gamma_t}  D^{H_\star}\Big(\nabla H(x_t) -  \frac{\alpha_t}{\gamma_t} \nabla f_{\xi}(x_t) ,\nabla H(x_t) \Big) \\
& + \frac{\beta_t}{\beta_t + \gamma_t} D^{H_\star}(\nabla H(\hat{x}_0),\nabla H(x_t))
\end{align*}
We also notice that
\begin{align*}
D^{H_\star}(\nabla H(\hat{x}_0),\nabla H(x_t)) = D^H(x_t,\hat{x}_0) & = H(x_t) - H(\hat{x}_0) - \langle \nabla H(\hat{x}_0),x_t\rangle \\
& = H(x_t) - H(\hat{x}_0)
\end{align*}
After simplification, we obtain
\begin{align*}
& \alpha_t  \Big( D^{F}(x_t,x_\star) + D^{F}(x_\star,x_t)  \Big) \\
\leq & \gamma_t D^{H}(x_\star, x_{t}) -  (\beta_t + \gamma_t) \mathbb{E} D^{H}(x_\star,x_{t+1}) +  \beta_t \Big( H(x_\star) -  H(\hat{x}_0)   \Big)    \\
+ &  \gamma_t \mathbb{E}   D^{H_\star}\Big(\nabla H(x_t) -  \frac{\alpha_t}{\gamma_t} \nabla f_{\xi}(x_t) ,\nabla H(x_t) \Big)
\end{align*}
The last term corresponds to the variance that we are trying to bound. If $ \frac{\alpha_t}{\gamma_t} \geq \min \Big( \frac{1}{2L_\star}, \frac{1}{2\eta_0}\Big)$, we simply use the assumption the quadratic upper bound $G$ of the stochastic noise
\begin{align*}
D^{H_\star}\Big(\nabla H(x_t) -  \frac{\alpha_t}{\gamma_t} \nabla f_{\xi}(x_t) ,\nabla H(x_t) \Big) \leq \frac{\alpha_t^2 G^2}{\gamma_t^2 }.  
\end{align*}
In the other case, we decompose in the following way, 
\begin{align*}
& D^{H_\star}\Big(\nabla H(x_t) -  \frac{\alpha_t}{\gamma_t} \nabla f_{\xi}(x_t) ,\nabla H(x_t) \Big)\\
\leq & \frac{1}{2}   D^{H_\star}\Big(\nabla H(x_t) -  \frac{2\alpha_t}{\gamma_t} \nabla f_{\xi_t}(x_\star) ,\nabla H(x_t) \Big) \\
+ & \frac{1}{2}   D^{H_\star}\Big(\nabla H(x_t) -  \frac{2 \alpha_t}{\gamma_t} \Big( \nabla f_{\xi}(x_t) - \nabla f_{\xi}(x_\star)\Big) ,\nabla H(x_t) \Big)
\end{align*}
The first term is related to the variance assumption, if $\frac{2\alpha_t}{ \gamma_t} \leq \frac{1}{\eta_0}$,
\begin{align*}
\frac{ \gamma_t}{2} \mathbb{E}  D^{H_\star}\Big(\nabla H(x_t) -  \frac{2\alpha_t}{\gamma_t} \nabla f_{\xi}(x_\star) ,\nabla H(x_t) \Big)  \leq \frac{ \gamma_t}{2} \times \Big( \frac{2\alpha_t}{ \gamma_t} \Big)^2 \sigma^2  = \frac{2\sigma^2 \alpha_t^2  }{ \gamma_t }
\end{align*}
The second term is related to the relative smoothness and the cocoercivity of $f_\xi$. By Lemma \ref{lem:cocoercivity}, if $\frac{2\alpha_t}{ \gamma_t} \leq \frac{1}{L_\star},$
\begin{align*}
\frac{\gamma_t}{2} \mathbb{E}    D^{H_\star}\Big(\nabla H(x_t) -  \frac{2 \alpha_t}{\gamma_t} \Big( \nabla f_{\xi}(x_t) - \nabla f_{\xi}(x_\star) \Big)  ,\nabla H(x_t) \Big) & \leq \frac{ \gamma_t}{2} \times \frac{2\alpha_t}{ \gamma_t} \mathbb{E} D^{f_{\xi}}(x_t,x_\star) \\
& = \alpha_t D^{F}(x_t,x_\star)
\end{align*}
Combining everything, we know 
\begin{equation}
\begin{aligned}
 \alpha_t   D^{F}(x_\star,x_t)  & \leq  \gamma_t D^{H}(x, x_{t}) -  (\beta_t + \gamma_t) \mathbb{E} D^{H}(x,x_{t+1}) +  \beta_t \Big( H(x_\star) -  H(\hat{x}_0)   \Big) \\
 &    + \begin{cases}
   \frac{ 2\sigma^2 \alpha_t^2  }{ \gamma_t }  \quad \text{if $ \frac{\alpha_t}{\gamma_t} \leq \min \Big( \frac{1}{2L_\star}, \frac{1}{2\eta_0}\Big)$} \\
    \frac{ G^2 \alpha_t^2  }{ \gamma_t } \quad \text{else}
    \end{cases}   
\end{aligned}
\end{equation}

Now we note $T_0 = \arg\max_{t} \Big\{ \frac{2 \alpha_t}{ \gamma_t} \geq \min\{\frac{1}{L}, \frac{1}{\eta_0}\} \Big\}.$ By the telescopic sum 
\begin{align*}
 \sum_{t=1}^{T} \alpha_t  D^{F}(x_\star,x_t) & =     \gamma_1 D^{H}(x_\star, x_{t})  +  \sum_t \beta_t \Big( H(x_\star) -  H(\hat{x}_{0})   \Big)     + \sum_{t=T_0}^{T} \frac{2\sigma^2 \alpha_t^2  }{ \gamma_t } + \sum_{t=1}^{T_0}  \frac{G^2\alpha_t^2 }{\gamma_t }  
\end{align*}
We finish the proof by considering the polynomial step-size $\alpha_t \propto t^n, \beta_t \propto \mu t^n, \gamma_{t} \propto \mu t^{n+1}/n$ with $n \geq 0$. We notice that
\begin{align*}
 \sum_{t=1}^{T_0}  \frac{\alpha_t^2 G^2}{\gamma_t } = \frac{G^2}{\mu} \Big( \sum_{t=1}^{T_0} t^{n-1} \Big) = \mathcal{O} \Big( \frac{G^2 T_0^n}{n \mu}   \Big) = \mathcal{O} \left( \frac{G^2}{n \mu}  \max \Big\{ \frac{L^n}{\mu^n}, \frac{\eta_0^n}{\mu^n }   \Big\} \right).
\end{align*}
\end{proof}
To illustrate the previous theorem and to compare the current result to other existing results, we assume convexity and choose concrete values for $(n,\mu)$.

\begin{corollary}
Under previous assumption, if $D^F(x_\star,\cdot)$ is convex and  $T\geq \max \Big( L_\star^2 , \eta_0^2 \Big) $, by tuning $(n,\mu) = \Big(3,\frac{1}{\sqrt{T}} \Big)$, our algorithm yields 
\begin{equation}
\begin{aligned}
   \mathbb{E}[ D^{F}(x_\star,x_{T+1}^{ag})]  = O  \left(  \frac{ D^{H}(x_F,x_1)}{T^{3.5}} + \frac{\sigma^2 + D^2}{\sqrt{T}}+ \frac{G^2}{T^{3}} \right) 
\end{aligned}
\end{equation}
\end{corollary}

\begin{proof}
We apply the previous Theorem with $n=3$ and $\mu = \frac{1}{\sqrt{T}}$. We apply Jensen's inequality and note that 
\begin{align*}
 \mu \geq \max \Big( \frac{L_\star}{T}, \frac{\eta_0}{T} \Big) \quad \Leftrightarrow T \geq \max \Big( L_\star^2 , \eta_0^2 \Big)
\end{align*}
\end{proof}

\section{Grid search}
\label{section-parallelization}
When the function value $F$ can be evaluated, Nemirovski's grid method offers a parameter-independent algorithmic approach. The method operates by running $K$ concurrent sessions over $N$ iterations, with each session using a distinct value of the regularization parameter $\mu$. The algorithm then selects the optimal solution from among these $K$ candidate outputs.

The standard implementation sets the number of parallel sessions logarithmically in relation to the total iterations: $K := O (\log(T))$. However, when rough estimates of the problem parameters are available, a more informed choice of $K$ may yield better performance. The algorithm presented below accommodates an arbitrary selection of $K$ to allow for such flexibility.

\begin{algorithm}[H]
\caption{Grid search Algorithm}
\label{Grid search Algorithm}
\begin{algorithmic}
\Require Total iteration number $T \geq 0$, number of parallel session $K \geq 1$, Starting point $x_1 \in \mathcal{X}$, algorithm $\mathcal{A}$
\State Find $N$ such that $T/2 \leq  K N  < T $
\State Consider initial start point: $x_{1}^{0}= x_1$
\For{$ k \in \left[ - \lfloor K/2 \rfloor + 1, \lfloor K/2 \rfloor - 1  \right]$}
\State Run Algorithm $\mathcal{A}$ with $N$ iterations and regularizing coefficient $\mu= 2^{k} $, obtain $x_{N+1}^{k}$ as output
\EndFor
\State Compute the index $k_{out} = \arg \max_{k} F(x_{N+1}^{k})$
\State Return $x_{out} = x_{N+1}^{k_{out}} $
\end{algorithmic}
\end{algorithm}

The general idea of the parallelization algorithm is to use an exponential range of regularizer values $\mu$ to find the optimal one. Before giving the full version, we present first a simplified version of the Grid search lemma.
\begin{lemma}
\label{lemma parallization simplified}
Consider an algorithm $\mathcal{A}$ that take as inputs the number of iterations $N$ and a regularization parameter $\mu > 0$, we denote by $x_{N}^{\mu}$ the output of the algorithm. We assume that the algorithm $\mathcal{A}$ satisfies for all $ N \geq 1, \mu >0 $,
\[ \mathbb{E}[F(x_{N}^{\mu})] - F_\star \leq \mu A + \frac{B}{\mu N} \]
with $A,B \geq 0$ being positive constants. For simplicity, we suppose that $\log(T)$ is an integer. If we know that
\begin{equation}
\label{eq-grid-T-big-1}
 T  \geq   \max \left\{\Big( \frac{A}{B}\Big)^{1/3} ,\Big( \frac{B}{A}   \Big)^{1/4} \right\}
\end{equation}
then running the Algorithm \ref{Grid search Algorithm} with $K = 4 \log(T)  $ will ensure that within $T$ iterations, the output $x_{out}$ satisfies
\begin{align*}
 \mathbb{E}[F(x_{out})] - F_\star =  O  \left( \sqrt\frac{AB \log(T)}{T} \right).
\end{align*}
\end{lemma}
\begin{proof}
Intuitively, the optimal value of $\mu$ is $\sqrt{\frac{B}{AN}}$. Therefore, our goal is to show there exists $k \in [- \log(T),\log(T)]$ (logarithm base $2$) such that
\begin{equation}
\label{eq-grid-search-good-mu1}
\frac{1}{2} \sqrt{\frac{B}{AN}}  \leq \mu = 2^{k} \leq  \sqrt{\frac{B}{AN}}.
\end{equation}
To show the existence of such a $k$, we need to check the two extreme values $\mu = 2^{2\log(T)}$ and $\mu = 2^{-2\log(T)}$ (recall that $T \propto KN =  4N \log(T)$)
\begin{align*}
& \text{Case } \mu = 2^{2\log T} \text{:}  \quad \mu \geq \frac{1}{2} \sqrt{\frac{B}{A N} }\Leftrightarrow T^2  \geq \frac{1}{2} \sqrt{\frac{B}{A N} } \Leftarrow T^5  \geq   \log(T) \Big( \frac{B}{A } \Big)  \Leftarrow T^4   \geq \frac{B}{A }   \\
& \text{Case } \mu = 2^{-2\log T} \text{:}  \quad \mu \leq  \sqrt{\frac{B}{A N} }\Leftrightarrow \frac{1}{T^2}  \leq \sqrt{\frac{B}{A N} } \Leftarrow  \frac{1}{T^4}  \leq \frac{B}{A N} \Leftarrow   \frac{A}{B}  \leq T^3
\end{align*}
And we notice that both conditions are satisfied when $T$ is large enough in \eqref{eq-grid-T-big-1}. Therefore, there exists $k$ verifying \eqref{eq-grid-search-good-mu1}.  ($\mu=2^k \approx \sqrt{\frac{B}{AN}}$) If we consider the output of the $k$-th run
\begin{align*}
 \mathbb{E} [F(x_N^k)] - F_\star \leq  \mu A + \frac{B}{\mu N }  & =  2^k A + \frac{B}{2^k N } \\
& \leq  \sqrt{\frac{AB}{N}} + 2 \sqrt{\frac{AB}{N}} = 3 \sqrt{\frac{AB}{N}},
\end{align*}
and we conclude by recalling $T = 4N \log(T)$ and $F(x_{out}) \leq F(x_N^k)$,
\begin{align*}
\mathbb{E} [F(x_{out})] - F_\star  \leq 3\sqrt{\frac{4AB \log(T)}{T}} =   6\sqrt{\frac{AB \log(T)}{T}}.
\end{align*}
\end{proof}
It is interesting to note that the condition \eqref{eq-grid-T-big-1} scales differently than the initial hypothesis yielding $\mu \propto \frac{1}{\sqrt{T}}$  \[  T  \geq   \max \left\{ A , B \right\}, \]
therefore the Grid search Algorithm could be very powerful in practice when we have a reliable way of comparing different values of $F$. Now, we present the Grid search Lemma where multiple terms could be involved.
\begin{lemma}
\label{lemma parallization}
Consider an algorithm $\mathcal{A}$ that take as inputs the number of iterations $N$ and a regularization parameter $\mu > 0$, we denote by $x_{N}^{\mu}$ the output of the algorithm. We assume that the algorithm $\mathcal{A}$ satisfies for all $ N \geq 1, \mu >0 $,
\[ \mathbb{E}[F(x_{N}^{\mu})] - F_\star \leq \mu A + \frac{B}{\mu^{\beta_1} N^{\alpha_1}} + \frac{C}{\mu^{\beta_2} N^{\alpha_2}} \]
with $A,B,C, \alpha_1,\alpha_2,\beta_1,\beta_2 \geq 0$ being positive constants. If the following condition is satisfied
\begin{equation}
\label{eq-grid-T-big-2}
\begin{aligned}
K + \frac{2\alpha_1 }{1+\beta_1} \log N \geq \frac{2}{1+\beta_1}  \log \Big(  \frac{2B }{A} \Big) + 2 \\
K + \frac{2\alpha_2 }{1+\beta_2} \log N \geq \frac{2}{1+\beta_2}  \log \Big(  \frac{2C }{A} \Big) + 2 \\
K  \geq 2 \min \left( \frac{1}{1+\beta_1} \log \Big(\frac{2AN^{\alpha_1}}{B}\Big),  \frac{1}{1+\beta_2} \log \Big(\frac{2AN^{\alpha_2}}{C}   \Big)  \right)
\end{aligned}
\end{equation}
then running Algorithm \ref{Grid search Algorithm} will ensure that within $T$ iterations, the output $x_{out}$ satisfies
\begin{align*}
 \mathbb{E}[F(x_{out})] - F_\star =  O  \left( \frac{ A^{\frac{\beta_1}{1+\beta_1}} B^{\frac{1}{1+\beta_1}} K^{\frac{\alpha_1}{1+\beta_1}}}{T^{\frac{\alpha_1}{1+ \beta_1}}}    +  \frac{A^{\frac{\beta_2}{1+\beta_2}} C^{\frac{1}{1+\beta_2}}K^{\frac{\alpha_2}{1+\beta_2}}}{  T^{\frac{\alpha_2}{1+ \beta_2}}} \right)
\end{align*}
\end{lemma}
\begin{proof}
The main difference with the previous Lemma \ref{lemma parallization simplified} is that multiple terms involving $\mu$ are in the equation and therefore there are multiple critical values of $\mu$ depending on $A,B,C,N$. 

In condition \eqref{eq-grid-T-big-2}, the first two lines mean that when $\mu$ is at its maximum value, $\mu A$ is the biggest term. In fact,
\begin{align*}
 \text{Case } \mu = 2^{\lfloor K/2 \rfloor - 1} \text{:}  \quad \mu A \geq\frac{B}{\mu^{\beta_1} N^{\alpha_1}} & \Leftrightarrow  
 \mu \geq  \frac{  B^{\frac{1}{1+\beta_1}} }{A^{\frac{1}{1+\beta_1}}N^{\frac{\alpha_1}{1+ \beta_1}}} \\
 & \Leftrightarrow 2^{\lfloor K/2 \rfloor - 1}  \geq  \frac{  B^{\frac{2}{1+\beta_1}} }{A^{\frac{1}{1+\beta_1}}N^{\frac{\alpha_1}{1+ \beta_1}}}  \\
 & \Leftrightarrow \lfloor K/2 \rfloor \geq \frac{1}{1+\beta_1}  \log \Big(  \frac{  2B }{A N^{\alpha_1}} \Big) \\
& \Leftarrow K \geq \frac{2}{1+\beta_1}  \log \Big(  \frac{  2B }{A N^{\alpha_1}} \Big) + 2   \\
& \Leftrightarrow K + \frac{2\alpha_1 }{1+\beta_1} \log N \geq \frac{2}{1+\beta_1}  \log \Big(  \frac{2B }{A} \Big) + 2   \\
\mu A \geq\frac{C}{\mu^{\beta_2} N^{\alpha_2}} & \Leftarrow  
K + \frac{2\alpha_2 }{1+\beta_2} \log N \geq \frac{2}{1+\beta_2}  \log \Big(  \frac{2C }{A} \Big) + 2 
\end{align*}
For the smallest value of $\mu$, we require one of the two following conditions to be satisfied (we recall that $T/2 \leq K N \leq T $):
\begin{align*}
 \text{Case } \mu = 2^{- \lfloor K/2 \rfloor + 1} \text{:}  \quad \mu A \leq\frac{B}{\mu^{\beta_1} N^{\alpha_1}}  & \Leftrightarrow 2^{- \lfloor K/2 \rfloor + 1}  \leq \frac{  B^{\frac{1}{1+\beta_1}} }{A^{\frac{1}{1+\beta_1}}N^{\frac{\alpha_1}{1+ \beta_1}}} \\
 & \Leftarrow \frac{1}{1+\beta_1}  \log \Big(  \frac{2A N^{\alpha_1}}{B} \Big) \leq K/2  \\
 \mu A \leq\frac{C}{\mu^{\beta_2} N^{\alpha_2}}  & \Leftarrow \frac{1}{1+\beta_2}  \log \Big(  \frac{2A N^{\alpha_2}}{C} \Big) \leq K/2  \\
\end{align*}
Since the dominant terms are different for the biggest and smallest values of $k$, we can consider the largest value of $k$ where $\mu A$ is not dominating the other two terms. In other words, we consider the largest $k$ such that one of the two following conditions is satisfied
\begin{equation}
\label{eq-lemma-parallelization-condition}
\begin{aligned}
\frac{1}{2} \Big( \frac{B}{AN^{\alpha_1}}\Big)^{\frac{1}{1+\beta_1}} \leq \mu= 2^{k} \leq \Big( \frac{B}{AN^{\alpha_1}}\Big)^{\frac{1}{1+\beta_1}} \\
\frac{1}{2} \Big( \frac{C}{AN^{\alpha_2}}\Big)^{\frac{1}{1+\beta_2}} \leq \mu= 2^{k} \leq \Big( \frac{C}{AN^{\alpha_2}}\Big)^{\frac{1}{1+\beta_2}} 
\end{aligned}
\end{equation}
Since it is the largest term not dominant, we know by replacing $\mu = 2^k$ by $2^{k+1}$, $\mu A$ will become dominant. Therefore, the other two terms with $B$ and $C$ cannot be much larger than $\mu A$. In other words, if one of the two previous conditions is satisfied with $\mu = 2^k$, we know
\begin{align*}
2^{k+1}  \geq  \frac{  B^{\frac{1}{1+\beta_1}} }{A^{\frac{1}{1+\beta_1}}N^{\frac{\alpha_1}{1+ \beta_1}}} \implies 2^{1+\beta_1} \mu^{1+\beta_1} \geq \frac{B}{AN^{\alpha_1}} \implies 2^{1+\beta_1}  \mu A \geq \frac{B}{\mu^{\beta_1} N^{\alpha_1}} \\
2^{k+1}  \geq  \frac{  C^{\frac{1}{1+\beta_2}} }{A^{\frac{1}{1+\beta_2}}N^{\frac{\alpha_2}{1+ \beta_2}}} \implies 2^{1+\beta_2} \mu^{1+\beta_2} \geq \frac{ C}{AN^{\alpha_2}} \implies 2^{1+\beta_2}  \mu A \geq \frac{C}{\mu^{\beta_2} N^{\alpha_2}}
\end{align*}
Therefore, we can consider the output of the $k$-th run
\begin{align*}
\mathbb{E} [F(x_N^k)] - F_\star  \leq & \mu A + \frac{B}{\mu^{ \beta_1} N^{\alpha_1}} +  \frac{C}{\mu^{\beta_2}N^{\alpha_2}}  \\
\leq & \mu A \left( 1 + 2^{1+\beta_1} + 2^{1+ \beta_2}\right)  
\end{align*}
Now if the first line of \eqref{eq-lemma-parallelization-condition} is satisfied, we have 
\begin{align*}
\mathbb{E} [F(x_N^k)] - F_\star & \leq    \left( 1 + 2^{1+\beta_1} + 2^{1+ \beta_2}\right)   A^{\frac{\beta_1}{1+\beta_1}} B^{\frac{1}{1+\beta_1}}   N^{-\frac{\alpha_1}{1+ \beta_1}}  \\
& \leq \left( 1 + 2^{1+\beta_1} + 2^{1+ \beta_2}\right)   A^{\frac{\beta_1}{1+\beta_1}} B^{\frac{1}{1+\beta_1}} K^{\frac{\alpha_1}{1+ \beta_1}}  T^{-\frac{\alpha_1}{1+ \beta_1}}
\end{align*}
If the second line of \eqref{eq-lemma-parallelization-condition} is satisfied, we have
\begin{align*}
\mathbb{E} [F(x_N^k)] - F_\star \leq \left( 1 + 2^{1+\beta_1} + 2^{1+ \beta_2}\right) A^{\frac{\beta_2}{1+\beta_2}} C^{\frac{1}{1+\beta_2}} K^{\frac{\alpha_2}{1+ \beta_2}}   T^{-\frac{\alpha_2}{1+ \beta_2}}  
\end{align*}
We conclude by merging the conditions, replacing the maximum of the two cases by their sum.
\end{proof}
We note that the previous Lemma could work if we have more than three terms in the equation, the reasoning would be similar by distinguishing different cases. We also note that grid search could work even if the original algorithm $\mathcal{A}$ does not have a guarantee for all $\mu >0$. The only important values of $\mu$ are the critical ones that we use in the previous proof.

For the rest of section, we will assume that is possible to evaluate the population function $F$ in an efficient way and we will run grid search $K= 4\log(T)$, and show the derived convergence rates. 

\textbf{Optimization in relative scale}
We recall the initial result from \ref{thm relative power smd},
\begin{equation}
\begin{aligned}
 \mathbb{E}  \left[ F(x_{T}^{ag})\right]  - F_\star   = O  \left( \mu  \Big(\frac{ F_\star}{\Gamma}  + \frac{H(\hat{x}_0) - H_\star }{T} \Big )   + \frac{\mathcal{M}n F_\star }{\mu T}   + \mathcal{L}\Big( \frac{\mathcal{M}n}{\mu T}\Big)^{n+1} + \frac{\mu D^{H}(x_F,x_{1}) }{T^{n+1}}  \right) 
\end{aligned}
\end{equation}
For the sake of simplicity, we would like to only choose the optimal value of $\mu$ to balance the coefficients in front of $F_\star$ ($\frac{\mu}{\Gamma} \approx \frac{\mathcal{M}n}{\mu T}$). The final result might not be optimal for the other terms, but general conditions are easier to satisfy. We can balance the value of $\mu $ around $\mu =\sqrt{\frac{\mathcal{M} n \Gamma}{T}}$, and we note that this value of $\mu$ satisfies $ 2\mathcal{M} n /T \leq \mu \leq \frac{\Gamma}{2} $. The new corollary becomes
\begin{corollary}
Assume that $H$ is $1$-strongly-convex with regard to $\| \cdot \|$ and there exist $\Gamma, \mathcal{M}, \mathcal{L}>0$ such that for all $x \in \mathcal{X}$,
\begin{align*}
\Gamma (H(x) - \min H) & \leq F(x) \\
\mathbb{E} \Big[ \| \nabla f_{\xi} (x, \xi)\|_{\ast}^2 \Big] & \leq  \mathcal{M}  F(x) \\
F(x) - F_\star & \leq \mathcal{L} 
\end{align*}
We assume $T$ is large enough,
\begin{align*}
T \geq \max \left\{ \Big(\frac{1}{\Gamma \mathcal{M} n } \Big)^{1/3}, \Big( \Gamma \mathcal{M} n \Big)^{1/4} \right\}
\end{align*}
Suppose Algorithm \ref{Algo Power SMD} runs under the oblivious step-size (\ref{oblivious-stepsize}) with the Grid search Algorithm \ref{Grid search Algorithm} and $K= 4 \lceil \log T \rceil$, then 
\begin{equation}
\begin{aligned}
 \mathbb{E}  \left[ F(x_{out})\right]  - F_\star   & = O  \left(  \log^{1/2}(T) \sqrt{\frac{\mathcal{M} n}{\Gamma T}} F_\star + \log^{3/2}(T) \sqrt{\mathcal{M} n \Gamma} \frac{H(\hat{x}_0) - H_\star }{T\sqrt{T}}  \right) \\
 & + O  \left(   \mathcal{L}\Big( \frac{\mathcal{M} n \Gamma\log(T)}{ T} \Big)^{ \frac{n+1}{2}} + \sqrt{\frac{\mathcal{M} n \Gamma D^{H}(x_F,x_{1}) }{ T}}  \Big(\frac{\log(T)}{T}\Big)^{n+1}  \right) 
\end{aligned}
\end{equation}
\end{corollary}
\begin{proof}
The proof strategy is similar to the one in Lemma \ref{lemma parallization simplified}. We are looking for $k$ to balance $\frac{\mu}{\Gamma}$ with $\frac{\mathcal{M}n}{\mu T}$ while neglecting other terms. More precisely, we look for $k$ such that
\begin{align*}
\frac{1}{2} \sqrt{\frac{\mathcal{M} n}{\Gamma T}} \leq \mu = 2^k \leq \sqrt{\frac{\mathcal{M} n}{\Gamma T}}
\end{align*}
and we choose this value of $\mu$ for other terms such as $\frac{\mu D^H(x_F,x_1)}{T^{n+1}} $.
\end{proof}

\textbf{Optimization with smooth function} We recall the result from \ref{thm smooth}, if $ 2 \sqrt{L/\mu} \leq T $, we know
\begin{align*}
 \mathbb{E} [F(x_{T+1}^{ag})]- F_\star  =O \left(  \mu D^2 +  \frac{n \sigma^2}{\mu T }+   \frac{LD^2}{T} \Big( \frac{Ln}{\mu T^2}\Big)^{n/2} + \frac{\mu  D^{H}(x_F,x_{1})}{nT^{n+1}} \right)  
\end{align*}
In the smooth case, there are two important values of $\mu$ to consider. To optimize the choice of $\mu$, we need to balance the values of $\mu$ around $\mu = \frac{\sigma}{D}\sqrt{\frac{n}{ T}}$ ($\mu D \approx \frac{n\sigma^2}{\mu T}$) and $\mu = \frac{L n^{\frac{n}{n+2}}}{T^{\frac{2(n+1)}{n+2}}}$ ($\mu \approx \frac{L}{T} ( \frac{Ln}{\mu T^2})^{n/2}$).
\begin{theorem} 
Assume $H$ is $1$-strongly-convex with regard to $\| \cdot \|$ and that there exists $D, D_H, L, \sigma >0$ such that for all $x,y \in \mathcal{X}$,
\begin{align*}
D^{F}(x,y) & \leq \frac{L}{2} \|x-y\|^2 \\
\mathbb{E}_{\xi} \Big[ \| \nabla f_{\xi}(x) - \nabla F(x) \|_{\ast}^2 \Big] & \leq \sigma^2 \\
\| x- y\|^2 & \leq D^2 \\
H(x_F) - H(x) & \leq D^2
\end{align*}
where $x_F$ is the minimum of $F$. We assume $T$ is large enough,
\begin{align*}
T & \geq \max \left \{ L^{\frac{2+n}{6+4n}} n^\frac{n}{6+4n}, \frac{n^{1/5} \sigma^{2/5}}{D^\frac{2}{5}} \right\} \\
T & \geq \max \left \{\frac{1}{L^{1+n/2} n^{n/2}}, \frac{D^2}{n \sigma^2}\right\} 
\end{align*}
Suppose Algorithm \ref{Algo ACSMD} runs under the oblivious step-size (\ref{oblivious-stepsize}) with the Grid search Algorithm \ref{Grid search Algorithm} and $K= 4 \lceil \log T \rceil $, then 
\begin{equation}
\begin{aligned}
 \mathbb{E}  \left[ F(x_{out})\right]  - F_\star & =  \mathcal{O} \left( \frac{LD^2 n^{\frac{n}{n+2}}\log^2(T)}{T^{\frac{2(n+1)}{n+2}}} + \frac{\sigma D \sqrt{n \log(T)}}{\sqrt{T}}  \right) \\
  & +  \mathcal{O} \left( \max \left\{ \frac{L \log^2(T)}{T^\frac{2(n+1)}{n+2}}, \frac{\sigma \log^{1/2}(T)}{D \sqrt{nT}} \right\}  \frac{ \log^{n+1}(T) D^{H}(x_F,x_{1})}{T^{n+1}} \right)
\end{aligned}
\end{equation}
\end{theorem}
\begin{proof}
The proof strategy is similar to the one in Lemma \ref{lemma parallization}. When $\mu = 2^{2\lceil \log T \rceil} \geq T^2$ , we verify that
\begin{align*}
\mu D^2 \geq  \frac{LD^2}{T} \Big( \frac{Ln}{\mu T^2}\Big)^{n/2} & \Leftrightarrow \mu^{2+n} T^{2+2n} \geq L^{2+n} n^{n} \Leftarrow T^{6+4n} \geq L^{2+n} n^{n} \\
\mu D^2 \geq \frac{n \sigma^2}{\mu T } & \Leftrightarrow \mu^2 T \geq \frac{n \sigma^2}{D^2} \Leftarrow T^5 \geq \frac{n \sigma^2}{D^2}
\end{align*}
When $\mu = 2^{-2\lceil \log T \rceil} \leq T^{-2}$ , we verify that
\begin{align*}
\mu D^2 \leq  \frac{LD^2}{T} \Big( \frac{Ln}{\mu T^2}\Big)^{n/2} & \Leftrightarrow \mu^{2+n} T^{2+2n} \leq L^{2+n} n^{n} \Leftarrow T^{2} \geq \frac{1}{L^{2+n} n^{n}} \\
\mu D^2 \leq \frac{n \sigma^2}{\mu T } & \Leftrightarrow \mu^2 T \leq \frac{n \sigma^2}{D^2} \Leftarrow T \geq \frac{D^2}{n \sigma^2}
\end{align*}
Since we supposed that $T$ is large enough to verify previous condition, now we choose the largest $k$ such that one of the two following conditions is verified
\begin{align*}
\frac{1}{2} \frac{L n^{\frac{n}{n+2}}}{T^{\frac{2(n+1)}{n+2}}} \leq \mu = 2^k \leq  \frac{L n^{\frac{n}{n+2}}}{T^{\frac{2(n+1)}{n+2}}} \\
\frac{1}{2} \frac{\sigma}{D}\sqrt{\frac{n}{ T}}  \leq \mu = 2^k \leq   \frac{\sigma}{D}\sqrt{\frac{n}{ T}} 
\end{align*}
If the first line is verified, since $\mu \geq \frac{4L}{T^2}$, we can apply the result from Theorem \ref{thm smooth},
\begin{align*}
\mathbb{E} [F(x_N^k)] - F_\star = \mathcal{O} \left( \frac{LD^2 n^{\frac{n}{n+2}}}{T^{\frac{2(n+1)}{n+2}}} + \frac{L D^{H}(x_F,x_{1})}{T^{n+1+ \frac{2(n+1)}{n+2}}} \right)
\end{align*}
If the second line is verified, we can deduce that :
\begin{align*}
\mu \geq \frac{1}{2} \frac{\sigma}{D}\sqrt{\frac{n}{ T}}  \geq \frac{L n^{\frac{n}{n+2}}}{T^{\frac{2(n+1)}{n+2}}} \geq \frac{4L}{T^2} 
\end{align*}
Therefore,
\begin{align*}
\mathbb{E} [F(x_N^k)] - F_\star = \mathcal{O} \left( \frac{\sigma D \sqrt{n}}{\sqrt{T}} + \frac{\sigma D^{H}(x_F,x_{1})}{D \sqrt{n} T^{n+3/2}} \right)
\end{align*}
We conclude the proof by taking the two bounds and add the poly-logarithmic terms. 
\end{proof}

\textbf{Optimization with relative smooth functions}
We recall the initial result from \ref{thm relative power smd}, if $\mu \geq \max \left\{ \frac{L_\star}{T}, \frac{\eta_0}{T} \right \}$
\begin{equation}
\begin{aligned}
   \frac{\sum_{t=1}^{T} n t^n \mathbb{E}[ D^{F}(x_\star,x_t)]}{T^{n+1}}   = \mathcal{O} \left(  \frac{\mu D^{H}(x_F,x_1)}{n T^{n+1}} + \mu D^2 + \frac{\sigma^2 n }{\mu T } + \frac{G^2}{n \mu T^{n+1}}  \max \Big\{ \frac{L^n}{\mu^n}, \frac{\eta_0^n}{\mu^n }   \Big\} \right) 
\end{aligned}
\end{equation}
For the sake of simplicity, we balance the value of $\mu =\frac{\sigma}{D}\sqrt{\frac{ n }{T}} $ only to equalize two terms $\mu D^2 \approx \frac{\sigma^2 n}{\mu T}$. This choice of $\mu$ is not optimal for other terms, but the conditions are easier to be satisfied. The new corollary becomes
\begin{corollary}
Assume that $H$ is strictly-convex, $D^F(x_\star,\cdot)$ is convex and there exist $\Gamma, \mathcal{M}, \mathcal{L}>0$ such that for all $x \in \mathcal{X}$,
\begin{align*}
D^F(x,y) & \leq L_\star D^H(x,y) \\
\forall \eta< \eta_0, \quad \mathbb{E}_\xi \left[ D^{H_\star}\Big( \nabla H(x) - 2\eta \nabla f_{\xi}(x_F), \nabla H(x)\Big) \right] & \leq 2\eta^2 \sigma^2 \\
\forall \eta, \quad \mathbb{E}_\xi \left[ D^{H_\star}\Big( \nabla H(x) - 2\eta \nabla f_{\xi}(x_F), \nabla H(x)\Big) \right] & \leq 2 \eta^2 G^2 \\
H(x_F) - \min H & \leq D^2 
\end{align*}
We assume $T$ is large enough,
\begin{align*}
T \geq \max \left\{ \Big(\frac{D^2}{\sigma^2 n } \Big)^{1/3}, \Big( \frac{\sigma^2 n}{D^2}\Big)^{1/4},  \frac{4L_\star^2 D^2}{\sigma^2 n }, \frac{4 \eta_0^2 D^2}{\sigma^2 n}  \right\} 
\end{align*}
Suppose Algorithm \ref{Algo Power SMD} runs under the oblivious step-size (\ref{oblivious-stepsize}) with the Grid search Algorithm \ref{Grid search Algorithm} and $K= 4 \lceil \log T \rceil$, then 
\begin{equation}
\begin{aligned}
 \mathbb{E}  \left[ F(x_{out})\right]  - F_\star  &  = O  \left(  \frac{\sigma D \sqrt{n} \log^{1/2}(T)}{\sqrt{T}}   + \frac{\sigma   \log^{n+1}(T)}{D\sqrt{n} T^{n+1/2}}D^{H}(x_F,x_1)\right)  \\
  & + O  \left(   \frac{G^2D \log^{2n}(T)}{ \sigma n \sqrt{n}T^{n+1/2}}  \max \Big\{ \frac{L \sigma}{D \sqrt{n T}}, \frac{\eta_0 \sigma}{D \sqrt{n T} }   \Big\}^n \right) 
\end{aligned}
\end{equation}
\end{corollary}
\begin{proof}
The proof strategy is similar to the one in Lemma \ref{lemma parallization simplified}. We are looking for $k$ such that
\begin{align*}
\frac{1}{2}\frac{\sigma}{D}\sqrt{\frac{ n }{T}} \leq \mu = 2^k \leq \frac{\sigma}{D}\sqrt{\frac{ n }{T}}
\end{align*}
The extra condition is to verify that when $T$ is large, we have 
\begin{align*}
\frac{1}{2} \frac{\sigma}{D}\sqrt{\frac{ n }{T}} \geq \max \left\{ \frac{L_\star}{T}, \frac{\eta_0}{T} \right\} \Leftrightarrow \sqrt{T} \geq \left\{ \frac{2L_\star D}{\sigma \sqrt{n}}, \frac{2 \eta_0 D}{\sigma \sqrt{n}} \right\}
\end{align*}
\end{proof}

\section{Numerical experiments}
For simulation, we are interested in minimizing the maximum eigenvalue function over a hypercube. We can apply both of our methods here, as the feasible set is bounded and the objective function is bounded over this feasible set as well.
\begin{equation}
\begin{aligned}
\min_{x \in \mathbb{R}^{d \times d} }&  \quad   \lambda_{\max}( x ) \\
s.t & \quad \| x - A \|_{\infty}  \leq \rho 
\end{aligned}
\end{equation}
In our simulation, to generate $A$, we will first generate random diagonal matrices then add Gaussian noise. When the Gaussian noise level is low, we refer to it as the sparse case. We mainly reproduce the setting from \cite{daspremont2014stochastic}, more details can be found in \ref{appendix-simulation}. We plot performance versus problem dimension and, having some flexibility on the degree of polynomials we are using in $\alpha_t \propto t^{n},\gamma_t \propto t^{n+1}/n$ for the oblivious step-sizes, we study how that degree affects algorithm performance in practice.

\subsection{Smoothed oracle}
In the first part of the experiments, we run simulations using the smoothed stochastic oracle suggested in \cite{daspremont2014stochastic}. This reference provides an estimation of the smoothness constant $L$, but the authors mention that it could be overestimating the parameter in practice. As a benchmark, we have run the method in \cite{GL12} using prior knowledge on $L$, and the Levy method \cite{levy2018online} using prior knowledge on $D$.
\begin{figure}[H]
    \label{graph-smooth}
    \centering
    \includegraphics[scale=0.28]{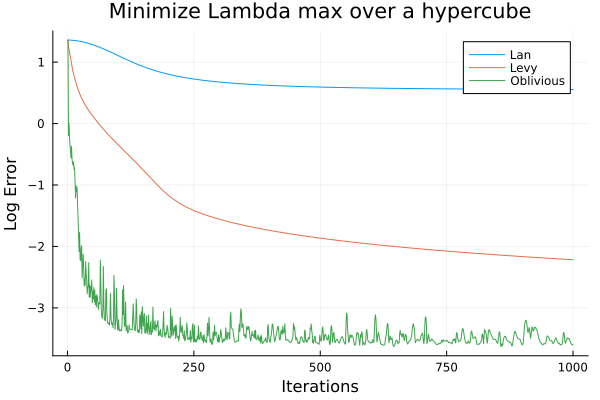}
    \includegraphics[scale=0.28]{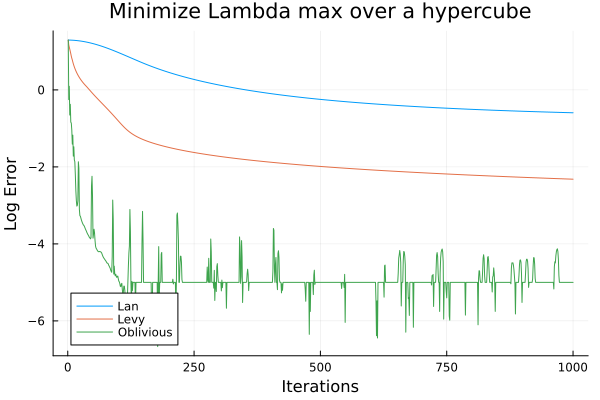}
    \includegraphics[scale=0.28]{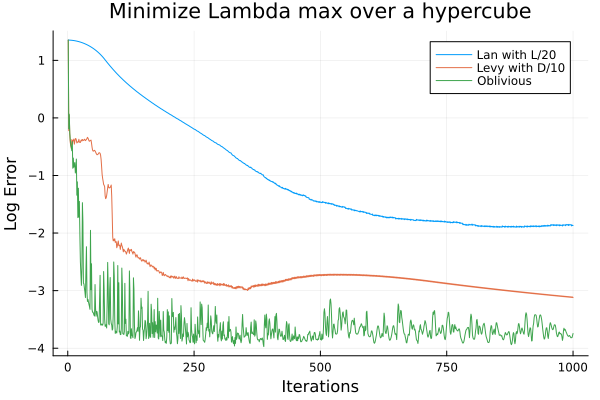}
    \includegraphics[scale=0.28]{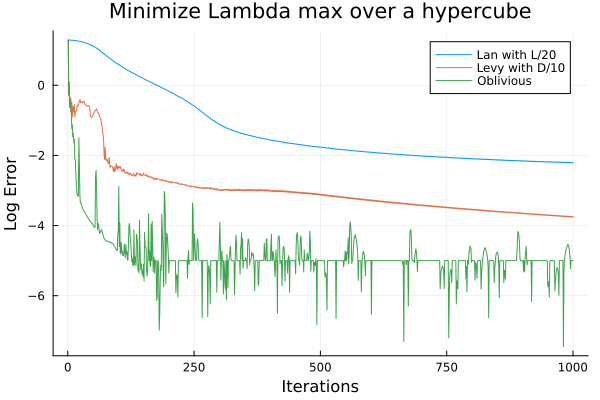}
    \caption{Performance on synthetic data, using stochastic smoothing in normal/sparse cases. Upper figures show performance under theoretical step-sizes, bottom figures use hyper-tuned step-sizes. }
\end{figure}

The first results in Fig.~\ref{graph-smooth} show that the current over-estimation of $L$ and $D$ slows down significantly both Lan and Levy's algorithms. In order to give an advantage to these methods, we also investigate the use of hyper-parameter tuning. 
We observe that the sparse case is well handled by our algorithms, whereas the Levy adaptive algorithm still keeps conservative step sizes. 

\begin{table}[H]
    \label{graph-smooth-2}
    \centering
    \begin{tabular}{|c | c | c | c | c | c |} 
     \hline
     Dimension & Lan & Levy & Oblivious1 & Oblivious2 & Oblivious3 \\ [0.5ex] 
     \hline\hline
     $d$=50 & $>1000$ & 125 & 22 & 29 & 39 \\
     \hline
     $d$=100 & $>1000$  & 529  & 45  & 57 & 70 \\
     \hline
     $d$=150  & $>1000$  & 679  & 106  & 163 & 244 \\ 
     \hline
    \end{tabular}
    \caption{Number of iterations required to reach the precision 1e-2 with stochastic smoothing.}
\end{table}


We notice that, at least in high dimensions, these algorithms' performance is dramatically improved by hyper-parameter tuning (see Fig.~\ref{graph-smooth-2}). In particular, our method still appears to be substantially faster than the alternatives.

\subsection{Power method oracle}
We reproduce these results using the power method oracle, with mirror descent in relative scale \cite{nesterov2023randomized} replacing Lan algorithm. The global phenomenon is similar as before, with $L_\star$ and $D$ over-estimated by the theory and we need to use hyper-tuned step-sizes for Levy and relative scale algorithms to obtain good performance.
\begin{figure}[H]
    \centering
    \includegraphics[scale=0.28]{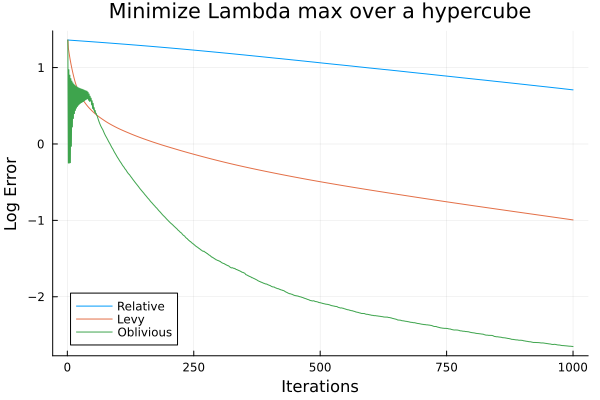}
    \includegraphics[scale=0.28]{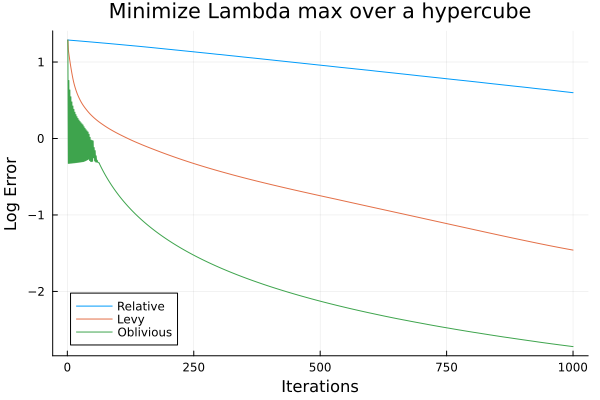}
    \includegraphics[scale=0.28]{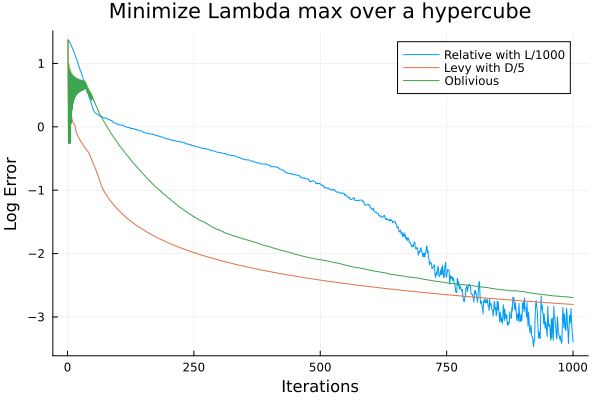}
    \includegraphics[scale=0.28]{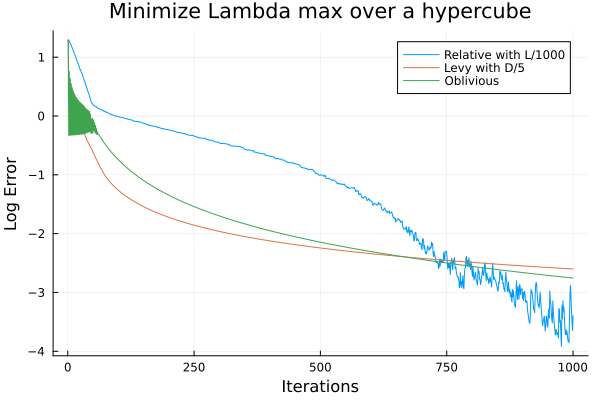}
    \caption{Performance of algorithms on synthetic data, using power method in the normal case and in the sparse case. The upper figures show performances when we use the parameters $L_\star$ and $D$ provided by the theory, the bottom one are using hyper-tuned step-sizes.} 
\end{figure}

However there are two major differences with the smooth case. (the objective function is not smooth anymore) First, we notice that our oblivious algorithm does not converge directly to the minimum as in the smooth case. Instead, our algorithm's first iterations are unstable and converges after a transition time, as our theory predicts. The other point we would like to underline is that power method oracles cost less to compute in practice, compared to the stochastic smoothing oracles, but the accuracy is also worse. More discussions and simulations results can be found in the supplementary material. \ref{appendix-simulation}

\vskip 6mm

\nocite{*}
\bibliographystyle{plain}
\bibliography{ref}

\newpage
\section{Simulations}
\label{appendix-simulation}
We recall the simulation problem that we are solving for the stochastic smoothing problem.
\begin{equation*}
\begin{aligned}
\min_{x \in \mathbb{R}^{d \times d} }&  \quad   \lambda_{\max}( x ) \\
s.t & \quad \| x - A \|_{\infty}  \leq \rho 
\end{aligned}
\end{equation*}
Note that, when using the power method oracle, we are actually solving an approximation of $\lambda_{\max}( x )^2$.

\textbf{Problem setting} To generate $A$, we chose a diagonal matrix $C$ with coefficients that decreases exponentially $C_{i,i} = \exp(-i)$, then we add Gaussian noise to each coefficients $A_{i,j} = C_{i,j} +  \mathcal{N}(0,\sigma^2) $ ($\sigma =0.2$ for the normal case and  $\sigma =0.05$ for the sparse case). We normalise and project to make $A$ symmetric and maximum coefficient equal to 1, and we chose $\rho = \max(diag(A))/2.$ For the stochastic smoothing oracle, we chose the target precision as $\epsilon = 1 e-2$. For the power method oracle, we use order $p=21$.

\textbf{Hyper-tuned step-sizes} For Lan, Levy and relative scale algorithm, we have run simulations with fewer iterations and lower dimension ($T=200, d = 50$). We discovered that by diminishing given constants $L$, $D$ or $L_\star$, all these algorithms require less iteration to reach $\epsilon = 1e-2$ precision. So for the rest of experiments, we have chosen some hyper-tuned coefficients $L/50$, $D/10$ or $L_{\star}/50$ (with associated hyper-tuned step-sizes) and we run these algorithms in higher dimension.

\textbf{Power method oracle in higher dimension}
The idea is to repeat the same experience as for the stochastic smoothed case, as the power method oracle costs less than the stochastic smoothing oracle, we have enlarged the number of iterations from $T=1000$ to $T=5000$. This will affect our oblivious algorithm since we choose $\mu = \sqrt{T}$.

\begin{table}[H]
\label{table-simulation-power}
    \centering
    \begin{tabular}{|c | c | c | c | c | c |} 
     \hline
     Dimension & Relative & Levy & Oblivious1 & Oblivious2 & Oblivious3 \\ [0.5ex] 
     \hline\hline
     $d$=50 & 1591 & 342 & 881 & 644 & 641 \\
     \hline
     $d$=100 & $>5000$  & 959  & 1172  & 713 & 629 \\
     \hline
     $d$=150 & $>5000$  & 2207  & 2659  & 1127  &  767  \\ 
     \hline
    \end{tabular}
    \vskip 1ex
    \caption{Number of iterations required to reach the precision 1e-2 with power method oracle for different algorithms.} 
\end{table}

On the previous table \ref{table-simulation-power}, we see that our oblivious algorithm is still efficient for problems in higher dimension. We also notice a major difference with the smooth case, in fact the degree of polynomial of oblivious algorithm plays a great role here, the oblivious algorithm using $\alpha_t \propto t^3$ require less iterations than the oblivious algorithm using $\alpha_t \propto t$. We think a good choice of the degree could be interesting in those situations.

\textbf{Execution time and stability of the results}
From a practical point of view, it is interesting to collect the computation time for these algorithms and see how the complexity evolves in higher dimension. It would also be interesting to know if the experience are easily reproducible, so we have repeated the simulations 10 times and put the results into the following table with the empirical mean and confidence interval.

\begin{table}[H]
    \centering
    \begin{tabular}{|c | c | c | c | |} 
     \hline
     Dimension $d=50$ & Lan & Levy & Oblivious \\ 
     \hline
     Mean & $> 6$ & $0.68$ & $ 0.14$ \\
     \hline
     Confidence interval & NA & $[0.51,1.03]$ &  $[0.12, 0.16]$ \\ [1ex] 
     \hline\hline
     Dimension $d=100$ & Lan & Levy & Oblivious \\ 
     \hline
     Mean & $> 18$ & $9.98$  & $0.8$ \\
     \hline
     Confidence interval & NA & $[7.73,12.6]$ & $[0.75, 0.88]$ \\ [1ex] 
     \hline\hline
     Dimension $d=150$ & Lan & Levy & Oblivious \\ 
     \hline
     Mean & $> 33$ & $31.2$ & $5.58$ \\
     \hline
     Confidence interval & NA & $ [19.6,>37]$ & $[4.9,6.25]$ \\ [1ex] 
     \hline
    \end{tabular}\vskip 1ex
    \caption{Empirical measures about computation time (in seconds) required to reach the precision 1e-2 with stochastic smoothing.}
\end{table}


     
\begin{table}[H]
    \centering
    \begin{tabular}{|c | c | c | c | |} 
     \hline
     Dimension $d=50$ & Relative & Levy & Oblivious3\\ 
     \hline
     Mean & $1.86$ & $0.39$ & $ 0.79$ \\
     \hline
     Confidence interval & $[1.80,1.92]$ & $[0.28,0.5]$ &  $[0.77,0.81]$ \\ [1ex] 
     \hline\hline
     Dimension $d=100$ & Relative & Levy & Oblivious3\\ 
     \hline
     Mean & $> 18$ & $4.9$  & $2.34$ \\
     \hline
     Confidence interval & NA & $[3.63,6.9]$ & $[2.25,2.4]$ \\ [1ex] 
     \hline\hline
     Dimension $d=150$ & Relative & Levy & Oblivious3\\ 
     \hline
     Mean & $> 26$ & $24.35$ & $4.0$ \\
     \hline
     Confidence interval & NA & $ [20.5, 26.9]$ & $[3.5,4.74]$ \\ [1ex] 
     \hline
    \end{tabular}\vskip 1ex
    \caption{Empirical measures about computation time (in seconds) required to reach the precision 1e-2 with power method.}
\end{table}

When the dimension increases, we notice that computational costs increase at non-linear speed. The reason is that the algorithm requires more computational time for each iteration and it requires more iterations in higher dimension to reach the precision 1e-2. We also observe that our oblivious algorithm outperforms the other two algorithms in higher dimension.

\end{document}